\documentclass[leqno,12pt, final]{amsart}
\usepackage{amssymb}
\usepackage{amsmath}
\usepackage{ccaption}
\usepackage{enumerate}
\usepackage{lscape}
\usepackage{multicol}
\usepackage{multirow}
\usepackage{tabularx}
\usepackage[all]{xy}
\theoremstyle{plain}
\newtheorem{theorem}{\indent\sc Theorem}[section]
\newtheorem{lemma}[theorem]{\indent\sc Lemma}
\newtheorem{corollary}[theorem]{\indent\sc Corollary}
\newtheorem{proposition}[theorem]{\indent\sc Proposition}
\newtheorem*{klemma}{\indent\sc Key Lemma}
\newtheorem*{mtheorem}{\indent\sc Theorem}
\theoremstyle{definition}
\newtheorem{definition}[theorem]{\indent\sc Definition}
\newtheorem{remark}[theorem]{\indent\sc Remark}

\newtheorem{recipe}[theorem]{\indent\sc Recipe}
\newcommand{\BS}[1]{\boldsymbol{#1}}
\newcommand{\MC}[1]{\mathcal{#1}}
\newcommand{\MF}[1]{\mathfrak{#1}}

\newcommand{\OPE}[1]{\operatorname{#1}} 

\newcommand{\RANK}{\operatorname{rank}}
\newcommand{\SRANK}{\operatorname{s\mathchar`-rank}}

\newcommand{\PreserveBackslash}[1]{\let\temp=\\#1\let\\=\temp}
\newcommand{\LISTV}[5]{#1&#2&#3&#4&#5\\\hline}
\newcommand{\LISTIV}[4]{#1&#2&#3&#4\\\hline}
\newlength{\LENGTHH}
\newlength{\LENGTHTHETA}
\newcommand{\LB}{\Large{$\bullet$}}

\usepackage{color}
\usepackage{comment}
\usepackage{ifdraft}
\usepackage{showkeys}
\usepackage{ulem}

\begin{document}

\title[Austere orbits for s-representations]{Examples of austere orbits of the isotropy representations for
semisimple pseudo-Riemannian symmetric spaces
}

\author[K. Baba]{Kurando Baba}

\subjclass[2010]{ 
Primary 53C40; Secondary 53C35.}

\keywords{ 
austere submanifold, pseudo-Riemannian symmetric space,
s-representation,
restricted root system.}

\address{
Department of General Education \endgraf
National Institute of Technology, Fukushima College \endgraf
Iwaki, Fukushima 970-8034 \endgraf
Japan
}
\email{baba@fukushima-nct.ac.jp}


\maketitle

\begin{abstract}
Harvey-Lawson and Anciaux introduced
the notion of austere submanifolds in pseudo-Riemannian geometry.
We give an equivalent condition for an orbit
of the isotropy representations for
semisimple pseudo-Riemannian symmetric space
to be an austere submanifold in a pseudo-sphere
in terms of restricted root system theory
with respect to Cartan subspaces.
By using the condition
we give examples of austere orbits.
\end{abstract}

\section*{Introduction} 

In pseudo-Riemannian geometry,
the notion of austere submanifolds
was introduced by Harvey-Lawson (\cite{MR0930601}) and Anciaux (\cite{MR2722116}).
They defined an austere submanifold as
a submanifold such that,
for each normal vector,
 the coefficients of odd degree
for the characteristic polynomial of its shape operator vanish.
In particular,
any austere submanifolds is a submanifold with vanishing mean curvature vector,
which is well-known as the minimal condition in Riemannian geometry.
Recently,
examples of austere submanifolds were given by using 
the method of orbits
on semisimple Riemannian symmetric spaces (\cite{MR2532897}, \cite{MR2752433}, \cite{MR3178477}).
In \cite{MR2532897},
Ikawa-Sakai-Tasaki classified austere orbits (in a sphere)
of the isotropy representation
for a semisimple Riemannian symmetric space
in terms of restricted root system theory.
The aim of this paper is to adapt their method
to a pseudo-Riemannian framework
and to give examples of austere orbits (in a pseudo-sphere)
of the isotropy representation for a semisimple pseudo-Riemannian symmetric spaces.

Let $G/H$
be a semisimple pseudo-Riemannian symmetric space
equipped with the metric induced from the Killing form $B$ of $\MF{g}$ ($:= \OPE{Lie}(G)$).
Let $\sigma$ be an involution of $\MF{g}$ whose fixed point set coincides with $\MF{h}$ ($:= \OPE{Lie}(H)$).
Denote by $\MF{q}$ the $(-1)$-eigenspace of $\sigma$,
which is identified with the tangent space of $G/H$ at the origin.
The isotropy representation of $G/H$
is equivalent to the adjoint representation $\OPE{Ad}$ of $H$ on $\MF{q}$.
Let $M$ be an $\OPE{Ad}(H)$-orbit through $X \in \MF{q}$.
If $X$ is non-null (i.e., $B(X,X) \neq 0$),
then $M$ is contained in the (central) hyperquadrics of $\MF{q}$.
In this paper,
we assume that $M$ is a pseudo-Riemannian submanifold
in the pseudo-hypersphere $\BS{S}$ ($:=\{v \in \MF{q} \mid B(v,v) = r(>0)\}$).
The nondegeneracy of the induced metric on $M \hookrightarrow \BS{S}$ implies the following result.

\begin{klemma}
Assume that the $\OPE{Ad}(H)$-orbit $M$ through $X \in \MF{q}$
is contained in a pseudo-hypersphere $\BS{S}$  $(\subset \MF{q})$.
Then,
$M \hookrightarrow \BS{S}$ is a pseudo-Riemannian submanifold
if and only if $X$ is semisimple $($i.e., an element of $\MF{q}$ such that $\OPE{ad}(X) \in \OPE{End}(\mathfrak{g})$
is diagonalizable over $\BS{C}$$)$.
\end{klemma}

\noindent
A main difficulty in the pseudo-Riemannian case
is the situation
that the shape operator is not diagonalizable over $\BS{C}$.
Therefore we give the Jordan-Chevalley decomposition
of the shape operator of $M \hookrightarrow \BS{S}$ (see, Proposition \ref{prop.jc}).
By using above Key Lemma
we describe the semisimple part and the nilpotent part
of the shape operator
in terms of restricted root system theory with respect to Cartan subspaces (cf.\ \cite{MR567427}, \cite{MR0930601} for the notion of restricted root system theory with respect to Cartan subspaces).
As its application,
we determine the spectrum of the shape operator (see, Corollary \ref{cor.spectrum}).
On the other hand,
in the Riemannian case,
any maximal abelian subspace is Cartan.
This implies that
all Cartan subspaces are mutually $\OPE{Ad}(H)$-conjugate (cf.\ \cite[Lemma 6.3, Chapter V]{MR1834454}).
However,
this conjugacy theorem does not necessarily hold
in the pseudo-Riemannian case.
Therefore we prove a conjugacy theorem
for complexified Cartan subspaces (see, Proposition \ref{prop.conj}).
By using these results we give an equivalent condition for $M \hookrightarrow \BS{S}$
to be austere (see, Proposition \ref{prop.cri}),
which is a generalization of Ikawa-Sakai-Tasaki's method (\cite{MR2532897}).
According to \cite{MR2532897},
the orbit through a restricted root vector
is an austere submanifold in a sphere.
In the pseudo-Riemannian case,
we need a technical condition
for restricted roots (see, Corollary \ref{cor.real}).
The main result of this paper is the following.

\begin{mtheorem}
For any restricted root $\alpha$ in Table \ref{table.realroots},
the $\OPE{Ad}(H)$-orbit through the restricted root vector of $\alpha$
is an austere submanifold in $\BS{S}$.
\end{mtheorem}

\begin{table}[!!h]
\scriptsize
\renewcommand{\arraystretch}{1.2}
\caption{The real restricted roots of $R$ with respect to a maximally split Cartan subspace}\label{table.realroots}
\centering
\newcolumntype{C}{>{\centering\arraybackslash}X}
\begin{tabularx}{\textwidth}{|c|C|}
\hline
Type of $(R, \theta)$ & Real Restricted Roots \\
\hline
\hline
AI                  & all restricted roots \\\hline
AIII                & $\{\pm(\alpha_{i}+\cdots+\alpha_{r+1-i}) \mid 1\leq i \leq l\}$  \\\hline
\multirow{2}{*}{BI} & $\{\pm(\alpha_{i}+\cdots+\alpha_{j-1}) \mid 1\leq i < j \leq l\}
\cup\{\pm(\alpha_{i}+\cdots+\alpha_{r}+\alpha_{j}+\cdots+\alpha_{r}) \mid 1\leq i < j \leq l\}$ \\
& $\cup\{\pm(\alpha_{i}+\cdots+\alpha_{r}) \mid 1\leq i \leq l\}$ \\\hline
\multirow{2}{*}{BCI} & $\{\pm(\alpha_{i}+\cdots+\alpha_{j-1}) \mid 1\leq i < j \leq l\}\cup\{\pm(\alpha_{i}+\cdots+\alpha_{r}+\alpha_{j}+\cdots+\alpha_{r}) \mid 1\leq i < j \leq l\}$ \\
& $\cup\{\pm(\alpha_{i}+\cdots+\alpha_{r}) \mid 1\leq i \leq l\}\cup\{\pm2(\alpha_{i}+\cdots+\alpha_{r}) \mid 1\leq i \leq l\}$ \\\hline
BCIII    & $\{\pm(\alpha_{2i-1}+2\alpha_{2i}+\cdots+2\alpha_{r}) \mid 1\leq i \leq l\}$ \\\hline
\multirow{2}{*}{CI} & $\{\pm(\alpha_{i}+\cdots+\alpha_{j-1}) \mid 1\leq i < j \leq l\}\cup\{\pm(\alpha_{i}+\cdots+\alpha_{j-1}+\alpha_{j}+\cdots+\alpha_{r}) \mid 1\leq i < j \leq l\}$\\
&$\cup\{\pm(2\alpha_{i}+\cdots+2\alpha_{r-1}+\alpha_{r}) \mid 1\leq i \leq l\}$ \\\hline
CIII & $\{\pm(\alpha_{2i-1}+2\alpha_{2i}+\cdots+2\alpha_{r-1}+\alpha_{r}) \mid 1 \leq i \leq l\}$ \\\hline
DI & $\{\pm(\alpha_{i}+\cdots+\alpha_{j-1}) \mid 1\leq i<j \leq l\}\cup \{\pm(\alpha_{i}+\cdots+\alpha_{r-2}+\alpha_{j}+\cdots+\alpha_{r}) \mid 1 \leq i < j \leq l\}$ \\\hline
DIII & $\{\pm(\alpha_{2i-1}+\cdots+\alpha_{r-2}+\alpha_{2i}+\cdots+\alpha_{r}) \mid 1\leq i \leq l\}$ \\\hline
\end{tabularx}
\end{table}

\begin{table}[!!h]
\scriptsize
\renewcommand{\arraystretch}{1.2}
\contcaption{(continued)}
\centering
\newcolumntype{C}{>{\centering\arraybackslash}X}
\begin{tabularx}{\textwidth}{|c|C|}
\hline
Type of $(R, \theta)$ & Real Restricted Roots \\
\hline
\hline
EI    & all restricted roots  \\\hline
\multirow{4}{*}{EII}   & 
$\left\{\pm\alpha_{2}, \pm\alpha_{4}, \pm(\alpha_{3}+\alpha_{4}+\alpha_{5}),
\pm(\alpha_{2}+\alpha_{4}),
\pm(\alpha_{2}+\alpha_{3}+\alpha_{4}+\alpha_{5})\right\}
\cup\left\{\pm(\alpha_{2}+\alpha_{3}+2\alpha_{4}+\alpha_{5}),
\pm(\alpha_{1}+\alpha_{3}+\alpha_{4}+\alpha_{5}+\alpha_{6}),
\pm(\alpha_{1}+\alpha_{2}+\alpha_{3}+\alpha_{4}+\alpha_{5}+\alpha_{6})\right\}
\cup\left\{\pm(\alpha_{1}+\alpha_{2}+\alpha_{3}+2\alpha_{4}+\alpha_{5}+\alpha_{6}),
\pm(\alpha_{1}+\alpha_{2}+2\alpha_{3}+2\alpha_{4}+2\alpha_{5}+\alpha_{6})\right\}
\cup\left\{\pm(\alpha_{1}+\alpha_{2}+2\alpha_{3}+3\alpha_{4}+2\alpha_{5}+\alpha_{6}),
\pm(\alpha_{1}+2\alpha_{2}+2\alpha_{3}+3\alpha_{4}+2\alpha_{5}+\alpha_{6})\right\}$ \\\hline
EIII  & 
$\{
\pm(\alpha_{1}+\alpha_{3}+\alpha_{4}+\alpha_{5}+\alpha_{6})
\pm(\alpha_{1}+2\alpha_{2}+2\alpha_{3}+3\alpha_{4}+2\alpha_{5}+\alpha_{6})
\}$ \\\hline
EV    & all restricted roots  \\\hline
\multirow{4}{*}{EVI}  & $\{
\pm \alpha_{1},
\pm \alpha_{3},
\pm(\alpha_{1}+\alpha_{3}),
\pm(\alpha_{2}+\alpha_{3}+2\alpha_{4}+\alpha_{5}),
\pm(\alpha_{1}+\alpha_{2}+\alpha_{3}+2\alpha_{4}+\alpha_{5})
\}
\cup\{
\pm(\alpha_{1}+\alpha_{2}+2\alpha_{3}+2\alpha_{4}+\alpha_{5}),
\pm(\alpha_{1}+\alpha_{2}+2\alpha_{3}+2\alpha_{4}+2\alpha_{5}+2\alpha_{6}+\alpha_{7})
\}
\cup\{
\pm(\alpha_{1}+\alpha_{2}+\alpha_{3}+2\alpha_{4}+2\alpha_{5}+2\alpha_{6}+\alpha_{7}),
\pm(2\alpha_{1}+2\alpha_{2}+3\alpha_{3}+4\alpha_{4}+3\alpha_{5}+2\alpha_{6}+\alpha_{7})
\}
\cup\{
\pm(\alpha_{1}+2\alpha_{2}+2\alpha_{3}+4\alpha_{4}+3\alpha_{5}+2\alpha_{6}+\alpha_{7}),
\pm(\alpha_{1}+2\alpha_{2}+3\alpha_{3}+4\alpha_{4}+3\alpha_{5}+2\alpha_{6}+\alpha_{7})
\}
\cup\{
\pm(\alpha_{2}+\alpha_{3}+2\alpha_{4}+2\alpha_{5}+2\alpha_{6}+\alpha_{7})
\}$\\\hline
EVII  & $\{
\pm \alpha_{7},
\pm(\alpha_{2}+\alpha_{3}+2\alpha_{4}+2\alpha_{5}+2\alpha_{6}+\alpha_{7}),
\pm(2\alpha_{1}+2\alpha_{2}+3\alpha_{3}+4\alpha_{4}+3\alpha_{5}+2\alpha_{6}+\alpha_{7})
\}$   \\\hline
EVIII & all restricted roots  \\\hline
\multirow{5}{*}{EIX}   &  
$\{
\pm \alpha_{7},
\pm \alpha_{8},
\pm(\alpha_{7}+\alpha_{8}),
\pm(\alpha_{2}+\alpha_{3}+2\alpha_{4}+2\alpha_{5}+2\alpha_{6}+\alpha_{7})
\}$\\
&$\cup\{
\pm(\alpha_{2}+\alpha_{3}+2\alpha_{4}+2\alpha_{5}+2\alpha_{6}+\alpha_{7}+\alpha_{8}),
\pm(\alpha_{2}+\alpha_{3}+2\alpha_{4}+2\alpha_{5}+2\alpha_{6}+2\alpha_{7}+\alpha_{8})
\}$\\
&$\cup\{
\pm(2\alpha_{1}+2\alpha_{2}+3\alpha_{3}+4\alpha_{4}+3\alpha_{5}+2\alpha_{6}+\alpha_{7})\}\cup\{\pm(2\alpha_{1}+2\alpha_{2}+3\alpha_{3}+4\alpha_{4}+3\alpha_{5}+2\alpha_{6}+2\alpha_{7}+\alpha_{8})
\}\cup\{
\pm(2\alpha_{1}+2\alpha_{2}+3\alpha_{3}+4\alpha_{4}+3\alpha_{5}+2\alpha_{6}+\alpha_{7}+\alpha_{8})\}\cup\{\pm(2\alpha_{1}+3\alpha_{2}+4\alpha_{3}+6\alpha_{4}+5\alpha_{5}+4\alpha_{6}+2\alpha_{7}+\alpha_{8})
\}\cup\{
\pm(2\alpha_{1}+3\alpha_{2}+4\alpha_{3}+6\alpha_{4}+5\alpha_{5}+4\alpha_{6}+3\alpha_{7}+\alpha_{8})\}\cup\{\pm(2\alpha_{1}+3\alpha_{2}+4\alpha_{3}+6\alpha_{4}+5\alpha_{5}+4\alpha_{6}+3\alpha_{7}+2\alpha_{8})
\}$\\\hline
FI    & all restricted roots   \\\hline
FII   & $\{
\pm(\alpha_{1}+2\alpha_{2}+3\alpha_{3}+2\alpha_{4})
\}$  \\\hline
FIII  & $\{
\pm(\alpha_{1}+\alpha_{2}+\alpha_{3}),
\pm(\alpha_{2}+2\alpha_{3}+2\alpha_{4}),
\pm(\alpha_{1}+2\alpha_{2}+3\alpha_{3}+2\alpha_{4}),
\pm(2\alpha_{1}+3\alpha_{2}+4\alpha_{3}+2\alpha_{4})
\}$ \\\hline
G     & all restricted roots   \\\hline
\end{tabularx}
\end{table}

\noindent
Here we remark on Theorem.
Let $\MF{a}$ be a Cartan subspace of $\MF{q}$
and $R$ ($\subset (\MF{a}^{\BS{C}})^{*} \setminus \{0\}$)  denote the restricted root system with respect to $\MF{a}$.
In the pseudo-Riemannian case,
a restricted root vector is in $\MF{a}$
if and only if its restricted root takes real values on $\MF{a}$.
In Theorem \ref{thm.realroots},
we classify all the real restricted roots
when $\MF{a}$ is maximally split
and the list is as in Table \ref{table.realroots}
(see, Section \ref{sec.pre} for
the definition of a maximally split Cartan subspace).
For the determination of the real roots,
we use a Satake diagram of $G/H$
associated with $(R, \theta)$,
where $\theta$ is a Cartan involution
of $\MF{g}$ such that $\theta$ commutes with $\sigma$
and preserves $\MF{a}$ invariantly
(cf.\ \cite{MR567427} for the existence of such a Cartan involution).
In Table \ref{table.realroots},
the types of $(R, \theta)$ are as in Table \ref{table.satakelistclassical},
the $\alpha_{i}$'s are fundamental roots as in Table \ref{table.satakelistclassical},
and $r$ (resp.\ $l$) denotes
the rank (resp.\ the split rank) of $G/H$.
In Table \ref{table.satakeclassical},
we determine 
the rank, the split rank and the type of $(R, \theta)$
for each irreducible pseudo-Riemannian symmetric space,
which was classified by Berger (\cite{MR0104763}).

The organization of this paper is as follows.
In Section \ref{sec.pre},
we prove Key Lemma,
give preliminaries for restricted root system theory
with respect to Cartan subspaces,
and recall the notion of its Satake diagram.
In Section \ref{sec.jcshape},
we give the Jordan-Chevalley decomposition for the shape operator
of an $\OPE{Ad}(H)$-orbit.
Moreover,
we determine the spectrum of the shape operator.
In Section \ref{sec.austere},
we prove Corollary \ref{cor.real} and Theorem \ref{thm.realroots},
which give the proof of Theorem.
In Appendix \ref{sec.appendix},
we give a recipe to determine
the Satake diagrams
associated with the restricted root systems
with respect to maximally split Cartan subspaces
for all semisimple (irreducible) pseudo-Riemannian symmetric spaces.

\indent\textsc{Future directions}.
We will classify all the austere orbits (in a pseudo-sphere)
of the isotropy representation
for a semisimple pseudo-Riemannian symmetric space.
For this purpose,
we need to determine the orbit space.
However,
the orbit space for general orbits
becomes quite complicated
in the pseudo-Riemannian case. 
We except that any austere
orbit is a hyperbolic orbit.
In \cite{B3},
the orbit space for hyperbolic orbits
is described in terms of restricted root system theory
with respect to maximal split abelian subspaces
(cf.\ \cite{MR518716}, \cite{MR810638} for the definition of a maximal split abelian subspace).

\indent\textsc{Acknowledgments}.
The author would like to express his sincere gratitude to
Professor Naoyuki Koike for valuable discussions and valuable comments.

\section{Preliminaries}\label{sec.pre}

Let $G$ be a connected semisimple noncompact Lie group,
$\sigma$ be an involution of $G$.
Let $H$ be a closed subgroup of $G$
with $(G_{\sigma})_{0} \subset H \subset G_{\sigma}$,
where $G_{\sigma}$ denotes the fixed point group of $\sigma$
and $(G_{\sigma})_{0}$ denotes its identity component.
The pair $(G, H)$ is called a \textit{semisimple symmetric pair}.
Then the coset space $G/H$ equipped with the metric induced from
the Killing form $B$ of $\MF{g}$ ($:=\OPE{Lie}(G)$)
is a semisimple pseudo-Riemannian symmetric space.
The involution $\sigma$ of $G$ induces an involution of $\MF{g}$,
which is also denoted by the same symbol $\sigma$.
Then the Lie algebra $\MF{h}$  of $H$ coincides with
$\{ X \in \MF{g} \mid \sigma(X)=X \}$.
The pair $(\MF{g}, \MF{h})$
is called a \textit{semisimple symmetric pair}.
Set $\MF{q}=\{ X \in \MF{g} \mid \sigma(X)=-X \}$,
which is identified with the tangent space of $G/H$
at the origin.
It is useful to identify
the isotropy representation of $G/H$
with the adjoint representation $\OPE{Ad}$
of $H$ on $\MF{q}$
in the context of symmetric spaces.
For each $X \in \MF{g}$,
the Jordan-Chevalley (JC) decomposition of $X$ is
induced from that of $\OPE{ad}(X) \in \OPE{End}(\MF{g})$
(cf. \cite[Proposition 1.3.5.1]{MR0498999}),
where $\OPE{ad}:\MF{g} \to \OPE{End}(\MF{g})$ is the adjoint representation of $\MF{g}$.
Denote by $X_{s}$ (resp.\ $X_{n}$)
the semisimple part (resp.\ the nilpotent part) of $X$.
By using Proposition 2 in \cite{MR567427}
we have $X_{s}, X_{n} \in \MF{q}$
if $X$ is in $\MF{q}$.
An element $X \in \MF{q}$ is said to be \textit{semisimple} (resp.\ \textit{nilpotent})
if $X=X_{s}$ (resp.\ $X=X_{n}$) holds.
Here,
we prove Key Lemma stated in Introduction.
\begin{proof}[\textsc{Proof of Key Lemma}]
Suppose that $M \hookrightarrow \BS{S}$ ($:=\{v \in \MF{q} \mid B(v,v)=r (>0)\}$)
 is a pseudo-Riemannian submanifold.
Let $X=X_{s}+X_{n}$ be the JC decomposition of $X$.
Since, for any $\xi \in T^{\perp}_{X}M$, $[\xi, X]=0$ holds,
we have $[\xi, X_{s}]=[\xi, X_{n}]=0$ by using Proposition 1.3.5.1 in
\cite{MR0498999}.
From Lemma 12 in \cite{MR567427}
there exists a $Z \in \MF{h}$ such that $[Z, X_{n}]=X_{n}$.
This implies that $X_{n}$ is orthogonal to $X$
by calculating
$B(X_{n},X) = B([Z,X_{n}],X)=B([X_{n}, X],Z)=0$.
Hence we have $X_{n} \in T^{\perp}_{X}M$.
The nondegeneracy of $M \hookrightarrow\BS{S}$
implies that the restriction of $B$ on $T^{\perp}_{X}M$ is nondegenerate.
By using the calculation $B(\xi,X_{n})= B(\xi,[Z, X_{n}])=B(Z,[X_{n},\xi])=0$ for all $\xi \in T^{\perp}_{X}M$,
we have $X_{n} = 0$.
Hence $X=X_{s}$ holds.
Conversely, let $X$ be a semisimple element in $\MF{q}$.
Then,
we have the eigenspace decomposition
$\MF{g}^{\BS{C}} = \sum_{\alpha \in \OPE{Spec}\OPE{ad}(X)}
\OPE{Ker}(\OPE{ad}(X)-\alpha \OPE{id})$ of
$\OPE{ad}(X) (\in \OPE{End}(\MF{g}^{\BS{C}}))$,
where $\OPE{Spec}\OPE{ad}(X)$ ($\subset \BS{C}$) denotes the spectrum
of $\OPE{ad}(X)$
and $\OPE{id}$ denotes the identity transformation on $\MF{g}^{\BS{C}}$.
Since $\sigma(\OPE{Ker}(\OPE{ad}(X) - \alpha \OPE{id}))
= \OPE{Ker}(\OPE{ad}(X) + \alpha \OPE{id})$,
we have a decomposition of $\MF{q}^{\BS{C}}$ as follows:
\begin{equation}\label{eq.adx}
\MF{q}^{\BS{C}} = \OPE{Ker}\OPE{ad}(X) \cap \MF{q}^{\BS{C}}
+ \sum_{\alpha \in \OPE{Spec}\OPE{ad}(X)\setminus\{0\}}
(\OPE{Ker}(\OPE{ad}(X)-\alpha\OPE{id})+\OPE{Ker}(\OPE{ad}(X)+\alpha\OPE{id}))
\cap\MF{q}^{\BS{C}}.
\end{equation}
Then we have $(T_{X}M)^{\BS{C}} = \sum_{\alpha \in \OPE{Spec}\OPE{ad}(X)\setminus\{0\}}
(\OPE{Ker}(\OPE{ad}(X)-\alpha\OPE{id})+\OPE{Ker}(\OPE{ad}(X)+\alpha\OPE{id}))
\cap\MF{q}^{\BS{C}}$,
where $(T_{X}M)^{\BS{C}}$ denotes the complexification of the tangent space of $M$ at $X$.
Since the decomposition (\ref{eq.adx})
is orthogonal with respect to $B$,
the restriction of $B$ on $(T_{X}M)^{\BS{C}}$
is nondegenerate.
Hence $M \hookrightarrow\BS{S}$ is a pseudo-Riemannian submanifold.
\end{proof}

\noindent
Note that,
any semisimple element in $\MF{q}$ is contained in a Cartan subspace of $\MF{q}$
(i.e., a maximal abelian subspace of $\MF{q}$ which consists of semisimple elements).
In the sequel,
we recall the notion of restricted root system theory
with respect to Cartan subspaces (cf.\ \cite{MR567427}, \cite{MR0930601}).
Let $\MF{a}$ be a Cartan subspace of $\MF{q}$.
Set,
for any $\alpha \in (\MF{a}^{\BS{C}})^{*}$,
\begin{align*}
\MF{g}^{\BS{C}}_{\alpha} &= \left\{X \in \MF{g}^{\BS{C}} \mid \OPE{ad}(A)X=\alpha(A)X, \forall A \in \MF{a}^{\BS{C}}\right\},\\
\MF{h}^{\BS{C}}_{\alpha} &= \left\{Z \in \MF{h}^{\BS{C}} \mid \OPE{ad}(A)^{2}Z=\alpha(A)^{2}Z, \forall A \in \MF{a}^{\BS{C}}\right\},\\
\MF{q}^{\BS{C}}_{\alpha} &=\left\{Y \in \MF{h}^{\BS{C}} \mid \OPE{ad}(A)^{2}Y=\alpha(A)^{2}Y, \forall A \in \MF{a}^{\BS{C}}\right\}.
\end{align*}
Denote by $R = \{ \alpha \in (\MF{a}^{\BS{C}})^{*} \setminus \{0\} \mid \MF{q}^{\BS{C}}_{\alpha} \neq \{0\}\}$,
which is called the \textit{restricted root system}
of $G/H$ (or $(\MF{g}, \MF{h})$) with respect to $\MF{a}$.
Then $R$ becomes a (reduced) root system.
For each $\alpha \in R$,
the dimension of $\MF{q}^{\BS{C}}_{\alpha}$ is called the \textit{multiplicity} of $\alpha$.
The dimension of $\MF{a}$
is called the \textit{rank} of $G/H$ (or $(\MF{g}, \MF{h})$).
Note that the type of $R$ (as root system)
and the value of $\OPE{rank}(G/H)$
do not depend on the choice of a Cartan subspace of $\MF{q}$.

\begin{lemma}[{\cite[2.1 Proposition]{MR0930601}}]
Assume that the $\OPE{Ad}(H)$-orbit $M$ through $X \in \MF{a}$ is contained in $\BS{S}$.
Then we have orthogonal decompositions of
$(T_{X}M)^{\BS{C}}$ and
the complexification of the normal space of $M$ in $\BS{S}$ as follows:
\begin{align*}
(T_{X}M)^{\BS{C}} &= \sum_{\alpha \in R_{+}:\ \alpha(X)\neq 0}\MF{q}^{\BS{C}}_{\alpha},\\
(T^{\perp}_{X}M)^{\BS{C}} &= (\MF{a}\ominus\BS{R}X)^{\BS{C}} + \sum_{\alpha \in R_{+}:\ \alpha(X)=0}\MF{q}^{\BS{C}}_{\alpha},
\end{align*}
where $R_{+}$ is a positive root system of $R$.
Moreover,
the above decompositions are orthogonal with respect to the Killing form of $\MF{g}^{\BS{C}}$.
\end{lemma}

\noindent
For each $\alpha \in R$,
we define a vector $A_{\alpha} \in \MF{a}^{\BS{C}}$
by $B(A, A_{\alpha})=\alpha(A)$ for all $A \in \MF{a}^{\BS{C}}$,
which is called the \textit{restricted root vector} of $\alpha$.
A restricted root $\alpha \in R$ is said to be \textit{real} (resp.\ \textit{imaginary})
if $\alpha$ takes real (resp.\ pure imaginary) values on $\MF{a}$.
It is clear that $A_{\alpha} \in \MF{a}$ (resp.\ $\sqrt{-1}A_{\alpha} \in \MF{a}$)
if and only if $\alpha$ is real (resp.\ imaginary).
For each semisimple pseudo-Riemannian symmetric space,
we will determine all the real restricted roots and all the imaginary restricted roots
(see, Section \ref{sec.austere}).
For this purpose,
we give a useful condition for $\alpha \in R$
to be real or imaginary by using a Cartan involution of $\MF{g}$.
Let $\theta$ be a Cartan involution of $\MF{g}$
such that $\theta$ commutes with $\sigma$
and preserves $\MF{a}$ invariantly
(cf.\ \cite[Lemma 5]{MR567427}).
Let $\MF{g}=\MF{k}+\MF{p}$ denote the Cartan decomposition
corresponding $\theta$.
Since $\MF{a}$ is $\theta$-invariant,
we have $\MF{a}=\MF{k}\cap\MF{a}+\MF{p}\cap\MF{a}$.
Then,
for each $\alpha \in R$,
$\alpha$ takes real values on $\sqrt{-1}(\MF{k}\cap\MF{a})+\MF{p}\cap\MF{a}$
($=:\MF{a}_{\BS{R}}\subset \MF{a}^{\BS{C}}$).
Hence $\alpha$ is real (resp.\ imaginary)
if and only if $\theta(\alpha)=-\alpha$ (resp.\ $\theta(\alpha)=\alpha$).
We can make use of a Satake diagram associated with $(R, \theta, \MF{a})$
to determine subsets $\{\alpha \in R \mid \theta(\alpha)=-\alpha\}$
and $\{\alpha \in R \mid \theta(\alpha)=\alpha\}$ ($=:R_{0}$)
of $R$.
Let $>$ denote the lexicographic ordering in $(\MF{a}_{\BS{R}})^{*}$
with respect to a basis $(A_{1}, \ldots, A_{l},A_{l+1}, \ldots, A_{r})$ of $\MF{a}_{\BS{R}}$
such that $(A_{1},\ldots,A_{l})$ (resp.\ $(A_{l+1},\ldots,A_{r})$) is a basis of $\MF{p}\cap\MF{a}$
(resp.\ $\sqrt{-1}(\MF{k}\cap\MF{a})$),
where $r=\OPE{rank}R$ and $l=\dim (\MF{p}\cap\MF{a})$.
Then the order $>$ becomes a $(-\theta)$-order in $R$ (cf.\ \cite{MR1280269}).
Denote by $\varPsi(R)$ the fundamental system of $R$ with respect to $>$.
Set $\varPsi(R_{0})=\varPsi(R)\cap R_{0}$.
Then we have the following result.

\begin{lemma}[{\cite[Theorem 5.4]{MR1280269}}]\label{lem.satakeinv}
There exists a permutation $p$ of $\varPsi(R)\setminus\varPsi(R_{0})$ with order $2$
such that, for each $\alpha \in \varPsi(R)\setminus\varPsi(R_{0})$,
$(-\theta)(\alpha)\equiv p\alpha$ $(\OPE{mod}\OPE{Span}_{\BS{Z}}\{\alpha \mid \alpha \in \varPsi(R_{0})\})$.
\end{lemma}

\noindent
We call the permutation $p$ as in Lemma \ref{lem.satakeinv}
the \textit{Satake involution} of $\varPsi(R)\setminus\varPsi(R_{0})$.
From the Dynkin diagram of $\varPsi(R)$ we define the Satake diagram associated with $(R, \theta, \MF{a})$
as follows.
First, replace a white circle of the Dynkin diagram,
which belongs to $\varPsi(R_{0})$,
with a black circle.
Next,
if restricted roots $\alpha, \beta \in \varPsi(R)\setminus\varPsi(R_{0})$
satisfy $\alpha \neq \beta$ and $p\alpha = \beta$,
join $\alpha$ and $\beta$ with an arrowed segment $\leftrightarrow$.
Note that this Satake diagram
depends on the choice of a Cartan subspace of $\MF{q}$.
A Cartan subspace $\MF{a}$ is said to be \textit{maximally split}
(resp.\ \textit{maximally compact}) if $\MF{p}\cap\MF{a}$ (resp.\ $\MF{k}\cap\MF{a}$)
is a maximal abelian subspace of $\MF{p}\cap\MF{q}$ (resp.\ $\MF{k}\cap\MF{q}$).
The dimension of the $\MF{p}$-part of a maximally split Cartan subspace (MSCS) is called
the \textit{split rank} of $G/H$ (or $(\MF{g}, \MF{h})$).
Any two MSCs are conjugate to each other.
Therefore, the definition of split rank does not depend on the choice of a MSCS.
We can easily determine the rank and the split rank
by using Table 2.5.2 in \cite{MR810638}.
In Table \ref{table.satakelistclassical},
we will determine
the rank, the split rank and
the Satake diagram associated with
$(R, \theta, \MF{a})$ when $\MF{a}$ is maximally split
for all semisimple pseudo-Riemannian symmetric spaces
(see, Appendix \ref{sec.appendix} for the determination).
Here we will often omit $\MF{a}$ for the notation of the Satake diagarm when there is no confusion.

\section{The Jordan-Chevalley decompositions of shape operators}\label{sec.jcshape}

Assume that the $\OPE{Ad}(H)$-orbit $M$ through an $X \in \MF{q}$
is a pseudo-Riemannian submanifold in a pseudo-hypersphere $\BS{S}$.
It follows from Key Lemma
that $X$ is semisimple.
In general,
the shape operator of $M \hookrightarrow \BS{S}$
is not necessarily diagonalizable over $\BS{C}$.
In this section,
for each $\xi \in T^{\perp}_{X}M$,
we give the JC decomposition of the shape operator $A_{\xi}$
in direction $\xi$,
where $A$ denotes the shape tensor of $M \hookrightarrow \BS{S}$.

\begin{lemma}\label{lem.normal}
Let $\xi$ be a normal vector of $M$ at $X$,
and $\xi = \xi_{s}+\xi_{n}$ be the JC decomposition
of $\xi$.
Then $\xi_{s}$, $\xi_{n}$ are normal vectors of $M$ at $X$.
\end{lemma}

\begin{proof}
By using 
Proposition 2 in \cite{MR567427} and
Proposition 1.3.5.1 in \cite{MR0498999}
we have $\xi_{s}, \xi_{n} \in \{Y \in \MF{q} \mid [Y, X]=0\}$.
From Lemma 12 in \cite{MR567427}
there exists a $Z \in \MF{h}$ such that $[Z, X_{n}]=X_{n}$.
Then we have
$B(\xi_{n}, X)=B([Z,\xi_{n}], X)=B([\xi_{n},X],Z)=0$,
where $B$ denotes the Killing form of $\MF{g}$.
Hence $\xi_{n} \in T^{\perp}_{X}M$ holds.
Moreover,
we have $\xi_{s}=\xi-\xi_{n} \in T^{\perp}_{X}M$.
\end{proof}

From above lemma
a decomposition
$A_{\xi}=A_{\xi_{s}}+A_{\xi_{n}}$
is well-defined.

\begin{proposition}\label{prop.jc}
Let $\xi=\xi_{s}+\xi_{n}$ be the JC decomposition of $\xi \in T^{\perp}_{X}M$.
Then the decomposition $A_{\xi} = A_{\xi_{s}}+A_{\xi_{n}}$
gives the JC decomposition of the shape operator $A_{\xi}$, i.e.,
$A_{\xi_{s}}$ is semisimple,
$A_{\xi_{n}}$ is nilpotent,
and $A_{\xi_{s}}A_{\xi_{n}}=A_{\xi_{n}}A_{\xi_{s}}$ hold.
\end{proposition}

The proof of Proposition \ref{prop.jc} requires some preparation.
For any $Z \in \MF{h}$,
we define a tangent vector field $Z^{*}$ on $M$
by $Z^{*}_{p}= (d/dt)|_{t=0}\OPE{Ad}(\exp tZ)p$ for all $p \in M$.
Then we have $A_{\xi}Z^{*}_{X}=-[Z, \xi]$ for all $\xi \in T^{\perp}_{X}M$.

\begin{lemma}\label{lem.semi}
For each semisimple $\xi \in T^{\perp}_{X}M$,
$A_{\xi}$ is semisimple.
Moreover, if $R$ the restricted root system
with respect to a Cartan subspace of $\MF{q}$
containing $X$ and $\xi$,
we have the spectrum of $A_{\xi}^{\BS{C}}$ as follows:
\begin{equation}\label{eq.speca}
\OPE{Spec}A_{\xi}^{\BS{C}}
=\left\{
-\dfrac{\alpha(\xi)}{\alpha(X)} \biggm| \alpha \in R_{+} \text{ with }\alpha(X) \neq 0
\right\}.
\end{equation}
\end{lemma}

\begin{proof}
Let $\MF{a}$ be a Cartan subspace of $\MF{q}$
containing $X$ and $\xi$,
and $R$ denote the restricted root system with respect to $\MF{a}$.
For each $\alpha \in R$ with $\alpha(X) \neq 0$,
we obtain $\MF{q}^{\BS{C}}_{\alpha} \subset \OPE{Ker}\left(A_{\xi}^{\BS{C}}+\dfrac{\alpha(\xi)}{\alpha(X)}\OPE{id}\right)$,
where $\OPE{id}$ denotes the identity transformation on $(T_{X}M)^{\BS{C}}$.
Therefore we have
\begin{equation*}
(T_{X}M)^{\BS{C}}
=\sum_{\alpha \in R_{+}: \alpha(X)\neq0}\MF{q}^{\BS{C}}_{\alpha}
\subset \sum_{\alpha \in R_{+}: \alpha(X)\neq0}\OPE{Ker}\left(A_{\xi}^{\BS{C}}+\dfrac{\alpha(\xi)}{\alpha(X)}\OPE{id}\right)
\subset (T_{X}M)^{\BS{C}},
\end{equation*}
where $R_{+}$ is a positive root system of $R$.
This implies that $A_{\xi}^{\BS{C}}$ is semisimple
and its spectrum is as in (\ref{eq.speca}).
\end{proof}

\begin{lemma}\label{lem.nilp}
For each nilpotent $\xi \in T^{\perp}_{X}M$,
$A_{\xi}$ is nilpotent.
\end{lemma}

\begin{proof}
Denote by $R$ the restricted root system with respect to a Cartan subspace
of $\MF{q}$ containing $X$.
For each $\alpha \in R$ with $\alpha(X) \neq 0$ and $n \in \BS{N}$
\begin{equation*}
(A_{\xi}^{\BS{C}})^{n}Y
=\begin{cases}
\dfrac{1}{\alpha(X)^{n}}\OPE{ad}(\xi)^{n}Y & (n:\text{even}),\\
-\dfrac{1}{\alpha(X)^{n+1}}\OPE{ad}(\xi)^{n}\OPE{ad}(X)Y & (n:\text{odd}),
\end{cases}
\end{equation*}
for all $Y \in \MF{q}^{\BS{C}}_{\alpha}$.
Since $\xi$ is nilpotent, we have $\OPE{ad}(\xi)^{n_{0}}=0$ for some integer
$n_{0} \in \BS{N}$.
Hence we have $(A_{\xi}^{\BS{C}})^{n_{0}}=0$,
i.e., $A_{\xi}$ is nilpotent.
\end{proof}

\begin{proof}[\indent\textsc{Proof of Proposition} $\ref{prop.jc}$]
Let $\xi=\xi_{s}+\xi_{n}$
be the JC decomposition of $\xi \in T^{\perp}_{X}M$.
It follows from Lemmas \ref{lem.semi} and \ref{lem.nilp}
that $A_{\xi_{s}}$ and $A_{\xi_{n}}$
are semisimple and nilpotent,
respectively.
Since $[\xi_{s},\xi_{n}]=0$ holds,
$A_{\xi_{s}}$ and $A_{\xi_{n}}$
commute with each other.
It follows from the uniqueness of the JC decomposition
of $A_{\xi}$ that
$A_{\xi_{s}}$ and $A_{\xi_{n}}$
coincide with the semisimple part and the nilpotent part
of $A_{\xi}$,
respectively.
\end{proof}

By using Proposition \ref{prop.jc} and Lemma \ref{lem.semi}
we have the following result.

\begin{corollary}\label{cor.spectrum}
Let $\xi = \xi_{s}+\xi_{n}$ be the JC decomposition of $\xi \in T^{\perp}_{X}M$,
and $R$ denote the restricted root system with respect to a Cartan subspace
of $\MF{q}$ containing $X$ and $\xi_{s}$.
The spectrum of $A_{\xi}^{\BS{C}}$ is given as follows:
\begin{equation*}
\OPE{Spec}A_{\xi}^{\BS{C}} =
\left\{
-\dfrac{\alpha(\xi_{s})}{\alpha(X)} \biggm| \alpha \in R \text{ with }\alpha(X) \neq 0
\right\}.
\end{equation*}
\end{corollary}

\section{Austere orbits}\label{sec.austere}

First, we recall the notion of austere submanifolds.

\begin{definition}[{\cite[Definition 3.15]{MR0666108}, \cite[p.\ 27]{MR2722116}}]
Let $\tilde{M}$ be a pseudo-Riemannian manifold.
A pseudo-Riemannian submanifold $M \hookrightarrow \tilde{M}$
is said to be \textit{austere} if,
for all $x \in M$ and $\xi \in T_{x}M$,
all the coefficients of odd degree for
the characteristic polynomial of $A_{\xi}$
vanish.
\end{definition}

\noindent
We can prove that
$M$ is austere if and only if,
for all $x \in M$ and $\xi \in T^{\perp}_{x}M$,
$\OPE{Spec}A_{\xi}^{\BS{C}}$
is invariant (considering multiplicities) under the multiplication by $-1$.
Therefore,
it is clear that
any austere submanifold has zero mean curvature.
In this section,
we will give an equivalent condition for
the austerity
when $M \hookrightarrow \BS{S}$ ($\subset \MF{q}$)
is an orbit of the isotropy representation
for a semisimple pseudo-Riemannian symmetric space,
which is identified with an $\OPE{Ad}(H)$-orbit
as we mentioned in Section \ref{sec.pre}.
Moreover,
we will give examples of austere orbits by using the condition.
Assume that the $\OPE{Ad}(H)$-orbit
$M$ through $X \in \MF{q}$
is a pseudo-Riemannian submanifold
in $\BS{S}$.
This implies that $X$ is a semisimple in $\MF{q}$.

\begin{proposition}\label{prop.cri}
Let $\MF{a}$ be a Cartan subspace of $\MF{q}$ containing $X$,
and $R$ denote the restricted root system with respect to $\MF{a}$.
Then $M$ is an austere orbit in $\BS{S}$
if and only if
$\{(-1/\alpha(X))p_{X}(\alpha) \mid \alpha \in R_{+} \text{ with } \alpha(X) \neq 0\}$ is invariant $($considering multiplicities$)$ under the multiplication by $-1$,
where $p_{X}$ denotes
the orthogonal projection along $X$.
\end{proposition}

\noindent
The proof of Proposition \ref{prop.cri}
requires some preparation.

\begin{lemma}\label{lem.ns}
Let $\MC{CS}_{X}$
denote the set of all Cartan subspaces of $\MF{q}$
containing $X$.
Then the set $(T^{\perp}_{X}M)_{s}$ of all semisimple normal vectors
in $T^{\perp}_{X}M$ is given as follows$:$
\begin{equation}\label{eq.snormal}
(T^{\perp}_{X}M)_{s}= \bigcup_{\MF{a} \in \MC{CS}_{X}}(\MF{a}\ominus\BS{R}X).
\end{equation}
\end{lemma}

\begin{proof}
Let $\xi$ be a semisimple normal vector in $T^{\perp}_{X}M$.
Then we have $[\xi, X]=0$ and $B(\xi, X)=0$.
Therefore,
there exists a Cartan subspace $\MF{a}$ of $\MF{q}$
containing $\xi$ and $X$.
Moreover, we have $\xi \in \MF{a}\ominus\BS{R}X$.
Conversely,
if $\xi$ is in the left side of (\ref{eq.snormal}),
then there exists a Cartan subspace of $\MF{q}$ satisfying
$X \in \MF{a}$ and $\xi \in \MF{a}\ominus\BS{R}X$.
This implies that $\xi$ is semisimple and commutes with $X$.
Hence we have $\xi \in (T^{\perp}_{X}M)_{s}$.
\end{proof}

\noindent
Let $\MF{g}_{X}$ denote the centralizer of $X$ in $\MF{g}$.
By imitating the proof of Proposition 1.3.5.4
in \cite{MR0498999}
we have the following result.

\begin{lemma}\label{lem.gx}
There exists a Cartan involution of $\MF{g}$
satisfying $\theta \circ \sigma = \sigma \circ \theta$
and $\theta(\MF{g}_{X})=\MF{g}_{X}$.
\end{lemma}

\begin{proof}
Let $\MF{a}$ be a Cartan subspace of $\MF{q}$ containing $X$.
Then there exists a Cartan involution of $\MF{g}$
satisfying $\theta \circ \sigma = \sigma \circ \theta$ and $\theta(\MF{a})=\MF{a}$ (cf.\ \cite[Lemma 5]{MR567427}).
If we put $\MF{k}=\OPE{Ker}(\theta - \OPE{id})$
and $\MF{p}=\OPE{Ker}(\theta + \OPE{id})$,
then $\MF{a}=\MF{k}\cap\MF{a}+\MF{p}\cap\MF{a}$ holds.
Denote by $R$ the restricted root system with respect to $\MF{a}$.
The restricted root space decomposition of $\MF{g}^{\BS{C}}$ gives
a decomposition
$\MF{g}^{\BS{C}}_{X}=\MF{g}^{\BS{C}}_{0}+\sum_{\alpha \in R_{X}}\MF{g}^{\BS{C}}_{\alpha}$
of $\MF{g}^{\BS{C}}_{X}$,
where $R_{X}:=\{\alpha \in R \mid \alpha(X)=0\}$.
Then we have $\MF{g}^{\BS{C}}_{0}$ is $\theta$-invariant.
If we write $X=X_{1}+X_{2}$ ($X_{1} \in \MF{k}\cap\MF{a}$,
$X_{2} \in \MF{p}\cap\MF{a}$),
then we have $\alpha(X_{i})=0$ ($i=1,2$)
for all $\alpha \in R_{X}$,
since $\alpha(\MF{k}\cap\MF{a}) \subset \sqrt{-1}\BS{R}$
and $\alpha(\MF{p}\cap\MF{a}) \subset \BS{R}$.
For any $\alpha \in R_{X}$,
we have $\theta\alpha(X)=\alpha(\theta(X))=\alpha(X_{1})-\alpha(X_{2})=0$.
This implies that $R_{X}$ is $\theta$-invariant.
Hence $\MF{g}^{\BS{C}}_{X}$ is $\theta$-invariant.
\end{proof}

\noindent
It follows from Lemma \ref{lem.gx}
that $\MF{g}_{X}$ is reductive in $\MF{g}$ (cf. \cite[Corollary 1.1.5.4]{MR0498999}).
On the other hand,
it is clear that $\MF{g}_{X}$ is $\sigma$-invariant.
Set $\MF{h}_{X}=\MF{h}\cap\MF{g}_{X}$
and $\MF{q}_{X}=\MF{q}\cap\MF{g}_{X}$.
Note that
any Cartan subspace of $\MF{q}$ containing $X$
is a Cartan subspace of $\MF{q}_{X}$
for a symmetric pair $(\MF{g}_{X}, \MF{h}_{X})$ (and vice versa).
Denote by $H_{X}$ the isotropy subgroup of $H$ at $X$.

\begin{lemma}
Let $\theta$ be a Cartan involution of $\MF{g}$
satisfying $\theta \circ \sigma = \sigma \circ \theta$
and $\theta(\MF{g}_{X})=\MF{g}_{X}$.
Then there exists a complete representatives
$\{\MF{a}_{1}, \ldots, \MF{a}_{m}\}$ for $\MC{CS}_{X}/H_{X}$
satisfying $\theta(\MF{a}_{i})=\MF{a}_{i}$
for $1 \leq i \leq m$.
\end{lemma}

\begin{proof}
Since $\MF{g}_{X}$ is reductive,
we have $\MF{g}_{X}=\MF{c}_{X}+\MF{s}_{X}$,
where
$\MF{c}_{X}$ (resp.\ $\MF{s}_{X}$) denotes the center
(resp.\ the semisimple part) of $\MF{g}_{X}$.
Then $\MF{c}_{X}$ and $\MF{s}_{X}$ are
invariant under the actions of $\sigma$ and $\theta$,
and $\MF{a}_{i}=\MF{c}_{X}\cap\MF{a}_{i}+\MF{s}_{X}\cap\MF{a}_{i}$ holds
for $1 \leq i \leq m$.
Moreover,
for $1 \leq i \leq m$,
we have $\MF{c}_{X}\cap\MF{a}_{i}=\MF{c}_{X}\cap\MF{q}$
and $\MF{s}_{X} \cap \MF{a}_{i}$ is a Cartan subspace of $\MF{s}_{X}\cap\MF{q}$
for $(\MF{s}_{X}, \MF{s}_{X}\cap\MF{h}_{X})$.
Since $\theta$ gives a Cartan involution of $\MF{s}_{X}$,
there exists an $h_{i} \in (H_{X})_{0}$
such that $\OPE{Ad}(h_{i})(\MF{s}_{X}\cap\MF{a}_{i})$ is $\theta$-invariant (cf.\ \cite[Remark]{MR567427})
and $\OPE{Ad}(h_{i})Y=Y$ for all $Y \in \MF{c}_{X}\cap\MF{q}$.
This implies that $\OPE{Ad}(h_{i})\MF{a}_{i}$ is $\theta$-invariant
for $1 \leq i \leq m$.
This proves the assertion above.
\end{proof}

\begin{proposition}\label{prop.conj}
Let $\theta$ be a Cartan involution of $\MF{g}$
satisfying $\theta \circ \sigma = \sigma \circ \theta$
and $\theta(\MF{g}_{X})=\MF{g}_{X}$.
Let $\{\MF{a}_{1}, \ldots, \MF{a}_{l}\}$
be a $\theta$-invariant complete representatives
for $\MC{CS}_{X}/H_{X}$.
For any $\MF{a}_{i}$ and $\MF{a}_{j}\ (1 \leq i \neq j \leq m)$,
there exists an isomorphism
$\psi$ on $\MF{g}^{\BS{C}}$
satisfying 
$\psi \circ \sigma = \sigma \circ \psi$,
$\psi(\MF{a}^{\BS{C}}_{i}\ominus \BS{C}X)=
\MF{a}^{\BS{C}}_{j}\ominus \BS{C}X$ and
$\psi(X)=X$.
\end{proposition}

\begin{proof}
Set $\MF{k}=\OPE{Ker}(\theta - \OPE{id})$,
$\MF{p}=\OPE{Ker}(\theta + \OPE{id})$,
$\MF{k}_{X}=\MF{k} \cap \MF{g}_{X}$ and
$\MF{p}_{X}=\MF{p} \cap \MF{g}_{X}$.
Then we have
$\MF{a}_{i} = \MF{k}\cap\MF{a}_{i}+\MF{p}\cap\MF{a}_{i}$,
$\MF{a}_{j} = \MF{k}\cap\MF{a}_{j}+\MF{p}\cap\MF{a}_{j}$,
and the
simultaneous decomposition
$\MF{g}_{X} = \MF{k}_{X} \cap \MF{h}_{X} + \MF{p}_{X} \cap \MF{h}_{X}
+\MF{k}_{X} \cap \MF{q}_{X}+\MF{p}_{X} \cap \MF{q}_{X}$
of $\MF{g}_{X}$ for $\sigma$ and $\theta$,
Set $\MF{g}^{d}_{X} = \MF{k}_{X} \cap \MF{h}_{X} + \sqrt{-1}(\MF{p}_{X} \cap \MF{h}_{X})
+\sqrt{-1}(\MF{k}_{X} \cap \MF{q}_{X})+\MF{p}_{X} \cap \MF{q}_{X} (\subset \MF{g}^{\BS{C}}_{X})$.
Then $\sigma$ gives a Cartan involution of $\MF{g}^{d}_{X}$,
so that $\MF{g}^{d}_{X}=\MF{k}^{d}_{X}+\MF{p}^{d}_{X}$
is the Cartan decomposition for $\sigma$,
where $\MF{k}^{d}_{X}:=\MF{k}_{X} \cap \MF{h}_{X} + \sqrt{-1}(\MF{p}_{X} \cap \MF{h}_{X})$
and $\MF{p}^{d}_{X}:=\sqrt{-1}(\MF{k}_{X} \cap \MF{q}_{X})+\MF{p}_{X} \cap \MF{q}_{X}$.
By the maximality of $\MF{a}_{i}$ (resp.\ $\MF{a}_{j}$) in $\MF{q}_{X}$
$\MF{a}^{d}_{i}:=\sqrt{-1}(\MF{k}\cap\MF{a}_{i})+\MF{p}\cap\MF{a_{i}}$
(resp.\ $\MF{a}^{d}_{j}:=\sqrt{-1}(\MF{k}\cap\MF{a}_{j})+\MF{p}\cap\MF{a}_{j}$)
is a maximal abelian subspace of $\MF{p}^{d}_{X}$.
This implies that
there exists some $Z_{1}, \ldots, Z_{k} \in \MF{k}^{d}_{X}$
satisfying $\MF{a}^{d}_{j}=e^{\OPE{ad}(Z_{1})}\cdots e^{\OPE{ad}(Z_{k})} \MF{a}^{d}_{i}$.
If we put $\psi=e^{\OPE{ad}(Z_{1})}\cdots e^{\OPE{ad}(Z_{k})}$,
then we have $\psi \circ \sigma = \sigma \circ \psi$ and $\psi(X)=X$.
Hence $\psi(\MF{a}^{\BS{C}}_{i}\ominus \BS{C}X)=
\MF{a}^{\BS{C}}_{j}\ominus \BS{C}X$ holds.
\end{proof}

By imitating the argument
in pp.\ 459--460, \cite{MR2532897}
we have the following result.

\begin{lemma}\label{lem.v}
Let $V$ be a vector space over $\BS{R}$
and $B$ be a nondegenerate bilinear form on $V$.
For any finite subset $\MC{A} \subset V^{\BS{C}}$,
the set $\{B(a, v) \mid a \in \MC{A}\}$ ($=:\MC{A}(v)\subset \BS{C}$)
is invariant by multiplication of $-1$ for all $v \in V$
if and only if $\MC{A}$ is invariant by multiplication of $-1$.
\end{lemma}

\begin{proof}
Suppose that
$\MC{A}(v)$ is invariant
by multiplication of $-1$ for all $v \in V$.
Take an $a \in \MC{A}$.
Then we have
\begin{equation*}
V = \bigcup_{b \in \MC{A}}\{v \in V \mid B(a,v)=-B(b,v)\}.
\end{equation*}
Since,
for each $b \in \MC{A}$,
$\{v \in V \mid B(a,v)=-B(b,v)\}$
is a subspace of $V$ and $\#A$ is finite,
there exists a $b_{0} \in \MC{A}$
with $B(a,v)=-B(b_{0},v)$ for all $v \in V$.
Then, for any $v=v_{1}+\sqrt{-1}v_{2} \in V^{\BS{C}}$ ($v_{1}, v_{2} \in V$)
we have $B(a, v)= B(a, v_{1})+\sqrt{-1}B(a,v_{2})
=-B(b_{0},v_{1})-\sqrt{-1}B(b_{0},v_{2})=-B(b_{0},v)$.
Since $B$ is nondegenerate on $V^{\BS{C}}$,
we have $-a =b_{0} \in \MC{A}$.
The converse is clear.
\end{proof}

\begin{proof}[\textsc{Proof of Proposition} $\ref{prop.cri}$]
It follows from Propositions \ref{prop.jc}, \ref{prop.conj}
and Lemma \ref{lem.ns}
that $M$ is austere
if and only if, for a Cartan subspace $\MF{a}$ and
all $\xi \in (\MF{a}\ominus \BS{R}X)$,
$\OPE{Spec}A^{\BS{C}}_{\xi}$
is invariant (considering multiplicities)
under the multiplication by $-1$.
By using the equation (\ref{eq.speca}) in Lemma \ref{lem.semi}
we have
\begin{equation*}
\OPE{Spec}A^{\BS{C}}_{\xi}
=\left\{B\left(-\dfrac{p_{X}(\alpha)}{\alpha(X)},\xi\right) \mid \alpha \in R_{+} \text{ with }\alpha(X) \neq 0\right\}
\end{equation*}
for all $\xi \in (\MF{a}\ominus \BS{R}X)$.
By applying Lemma \ref{lem.v} for 
$\MF{a}\ominus\BS{R}$ and
$\{(-1/\alpha(X))p_{X}(\alpha) \mid \alpha \in R_{+} \text{ with }\alpha(X) \neq 0\}$
($(\MF{a}\ominus\BS{R})^{\BS{C}}$)
we complete the proof.
\end{proof}

\begin{corollary}\label{cor.real}
The orbit through any real restricted root vector
is an austere submanifold in $\BS{S}$.
\end{corollary}

\begin{proof}
Let $\alpha$ be a real restricted root.
Then the restricted root vector $A_{\alpha}$ is in $\MF{q}$ and $B(A_{\alpha}, A_{\alpha})>0$.
If we put $X=A_{\alpha}$,
then
$\{(-1/\beta(X))p_{X}(\beta) \mid \beta \in R_{+} \text{ with }\beta(X) \neq 0\}$
($=:\MC{A}$)
is invariant $($considering multiplicities$)$ under the multiplication by $-1$.
Indeed,
for any $v = (-1/\beta(X))p_{X}(\beta) \in \MC{A}$,
we have
$s_{\alpha}(\beta)(X)\neq 0$ and
$-v=(-1/s_{\alpha}(\beta)(X))p_{X}(s_{\alpha}(\beta)) \in \MC{A}$.
\end{proof}

\begin{remark}
Ikawa-Sakai-Tasaki proved Corollary \ref{cor.real} 
in the case when $G/H$ is a Riemannian symmetric space (cf.\ \cite[Proposition 4.4]{MR2532897}).
In fact, they classified austere orbits (cf.\ \cite[Theorem 5.1]{MR2532897}).
\end{remark}

In the sequel,
we give all the real restricted roots in the restricted root system
with respect to a MSCS.
Let $\theta$ be a Cartan involution of $\MF{g}$ commuting with $\sigma$ (cf.\ \cite[Theorem 2.1, Chapter IV]{MR0239005}).
Denote by $R$ the restricted root system with respect to a $\theta$-invariant MSCS $\MF{a}$.
As we mentioned in Section \ref{sec.pre},
a restricted root $\alpha$ is real if and only if $\theta(\alpha)=-\alpha$.
On the other hand,
we can determine the action of $\theta$ on $R$
in terms of the Satake diagram associated with $(R, \theta)$.
Then we have the following result.

\begin{theorem}\label{thm.realroots}
All the real restricted roots in the restricted root system
with respect to a MSCS for all semisimple pseudo-Riemannian symmetric spaces are as in Table \ref{table.realroots}.
\end{theorem}

The proof of Theorem \ref{thm.realroots}
is given by Lemmas \ref{realAI}--\ref{realanti} as shown in the following.
Set $\tilde{\theta}=-\theta$ and $\alpha^{\tilde{\theta}}=-\theta(\alpha)$.
\begin{lemma}\label{realAI}\label{realDI-1}\label{EI}\label{EV}\label{EVIII}\label{FI}\label{G}
In the case where $(R, \theta)$ is of type
AI, DI$(\RANK=\SRANK)$,
EI, EV, EVIII, FI, or G,
all restricted roots are real.
\end{lemma}
\begin{proof}
From the Satake diagram of $(R, \theta)$
the Satake involution is trivial and $\varPsi(R_{0})=\emptyset$.
This implies that $\alpha^{\tilde{\theta}} = \alpha$
for all $\alpha \in \varPsi(R)$.
Therefore all restricted roots are real.
\end{proof}
In the sequel,
for each root $\alpha \in R$,
we give the form $\alpha = \sum n_{i}\alpha_{i}$,
where the $\alpha_{i}$'s are fundamental roots as in Table \ref{table.satakelistclassical}
and
the $n_{i}$'s are integers which are either all positive or all negative.
\begin{lemma}\label{realAII}
In the case where $(R, \theta)$ is of type AII,
there exists no real restricted root.
\end{lemma}
\begin{proof}
Without loss of generality, we assume that $\RANK R = 2r-1$.
From the Satake diagram of $(R,  \theta)$
the Satake involution is trivial and
$\varPsi(R_{0})=\{\alpha_{2i-1} \mid 1 \leq i \leq r\}$.
Note that any restricted root $\alpha$ is the form $\pm(\alpha_{i}+\cdots+\alpha_{j-1})$ for $1 \leq i < j \leq 2r$.
Therefore, for each $1 \leq i \leq r$,
the possibility of the form $\alpha^{\tilde{\theta}}_{2i}$ is
either $\alpha_{2i}$, $\alpha_{2i-1}+\alpha_{2i}$, $\alpha_{2i}+\alpha_{2i+1}$ or
$\alpha_{2i-1}+\alpha_{2i}+\alpha_{2i+1}$.
If $\alpha^{\tilde{\theta}}_{2i}=\alpha_{2i}$,
we have
$
(\alpha_{2i-1}+\alpha_{2i}+\alpha_{2i+1})^{\tilde{\theta}}
= -\alpha_{2i-1}+\alpha_{2i}-\alpha_{2i+1}
$.
But this contradicts that $(\alpha_{2i-1}+\alpha_{2i}+\alpha_{2i+1})^{\tilde{\theta}}$
is a restricted root.
Hence we have $\alpha^{\tilde{\theta}}_{2i}\neq\alpha_{2i}$.
By the same argument we have
$\alpha^{\tilde{\theta}}_{2i} = \alpha_{2i-1}+\alpha_{2i}+\alpha_{2i+1}$
for $1 \leq i \leq r-1$.
Moreover, we have
\begin{equation*}
(\alpha_{i}+\cdots+\alpha_{j-1})^{\tilde{\theta}}=
\begin{cases}
\alpha_{i+1}+\cdots+\alpha_{j} & (i:\text{odd},\ j:\text{odd}),\\
\alpha_{i+1}+\cdots+\alpha_{j-2} & (i:\text{odd},\ j:\text{even}),\\
\alpha_{i-1}+\cdots+\alpha_{j} & (i:\text{even},\ j:\text{odd}),\\
\alpha_{i-1}+\cdots+\alpha_{j-2} & (i:\text{even},\ j:\text{even}).
\end{cases}
\end{equation*}
Hence there exists no real restricted root.
\end{proof}
\begin{lemma}\label{realAIII}
In the case where $(R, \theta)$ is of
type AIII$(\RANK=r, \SRANK=l)$,
the set of all real restricted roots of $R$ coincides with
$\{\pm(\alpha_{i}+\cdots+\alpha_{r+1-i}) \mid 1\leq i \leq l\}$. 
\end{lemma}
\begin{proof}
In this case we have $p\alpha_{i}=\alpha_{r+1-i}$ for $i=1,\ldots,l,r-l+1,\ldots,r$.
First, we consider the case of $r=2l-1$ or $2l$.
Then, from the Satake diagram of $(R, \theta)$
we have $\Psi(R_{0})=\emptyset$.
This implies that $\alpha^{\tilde{\theta}}_{i}=p\alpha_{i}$
for $i=1,\ldots,l,r-l+1,\ldots,r$.
Therefore, for each $\alpha=\alpha_{i}+\cdots+\alpha_{j-1}$,
$\alpha^{\tilde{\theta}}=\alpha$ holds if and only if
$i+j=r+2$ holds.
Next,
we consider the case of $r>2l$.
From the Satake diagram of $(R, \theta)$ we have
$\varPsi(R_{0}) = \{\alpha_{i} \mid l+1\leq i\leq r-l\}$.
Since $\tilde{\theta}$ leaves $R$ invariant,
we have
\begin{equation*}
\alpha^{\tilde{\theta}}_{i}=
\begin{cases}
\alpha_{r-i+1} & (1 \leq i \leq l-1, r-l+2 \leq i \leq r),\\
\alpha_{l+1} + \cdots + \alpha_{r-l+1} & (i=l),\\
\alpha_{l} + \cdots + \alpha_{r-l} & (i=r-l+1),\\
-\alpha_{i} & (l+1 \leq i \leq r-l).
\end{cases}
\end{equation*}
Hence $\alpha^{\tilde{\theta}}=\alpha$ holds
if and only if
$\alpha$ has the form $\alpha=\pm(\alpha_{i}+\cdots+\alpha_{r+1-i})$.
\end{proof}
\begin{lemma}\label{realBI}
In the case where $(R, \theta)$ is of type
BI$(\RANK=r, \SRANK=l)$,
the set of all real restricted roots of $R$
coincides with
\begin{align*}
\{&\pm(\alpha_{i}+\cdots+\alpha_{j-1}) \mid 1\leq i < j \leq l\}\\
&\cup
\{\pm(\alpha_{i}+\cdots+\alpha_{r}+\alpha_{j}+\cdots+\alpha_{r}) \mid 1\leq i < j \leq l\}
\cup\{\pm(\alpha_{i}+\cdots+\alpha_{r}) \mid 1\leq i \leq l\}.
\end{align*}
\end{lemma}
\begin{proof}
From the Satake diagram of $(R, \theta)$
the Satake involution is trivial
and $\varPsi(R_{0})=\{\alpha_{l+k} \mid 1 \leq k \leq r-l\}$.
Since any positive root has the form
$\alpha_{i}+\cdots+\alpha_{j-1}, \alpha_{i}+\cdots+\alpha_{r}+\alpha_{j}+\cdots+\alpha_{r}$ or $\alpha_{i}+\cdots+\alpha_{r}$,
we have $\alpha^{\tilde{\theta}}_{i}=\alpha_{i}$ for $1 \leq i \leq l-1$.
Moreover,
$\alpha^{\tilde{\theta}}_{l}=\alpha_{l}+\cdots+\alpha_{r}+\alpha_{l+1}+\cdots+\alpha_{r}$ holds because $\tilde{\theta}$ leaves the root system $R$ invariant.
Therefore, by direct calculation
we can explicitly determine the set $\{\alpha \in R \mid \alpha = \alpha^{\tilde{\theta}}\}$ as in the assertion. 
\end{proof}
\noindent
By imitating the proof of Lemma \ref{realBI}
we have the following two facts.
\begin{lemma}\label{realBCI}
In the case where $(R, \theta)$ is of type
BCI$(\RANK=r, \SRANK=l)$,
the set of all real restricted roots of $R$
coincides with
\begin{align*}
\{&\pm(\alpha_{i}+\cdots+\alpha_{j-1}) \mid 1\leq i < j \leq l\}\\
&\cup\{\pm(\alpha_{i}+\cdots+\alpha_{r}+\alpha_{j}+\cdots+\alpha_{r}) \mid 1\leq i < j \leq l\}\\
&\cup
\{\pm(\alpha_{i}+\cdots+\alpha_{r}) \mid 1\leq i \leq l\}
\cup
\{\pm2(\alpha_{i}+\cdots+\alpha_{r}) \mid 1\leq i \leq l\}.
\end{align*}
\end{lemma}
\begin{lemma}\label{realCI}
In the case where $(R, \theta)$ is of type
CI$(\RANK=r, \SRANK=l)$,
the set of all real restricted roots
coincides with
\begin{align*}
\{&\pm(\alpha_{i}+\cdots+\alpha_{j-1}) \mid 1\leq i < j \leq l\}\\
&\cup
\{\pm(\alpha_{i}+\cdots+\alpha_{j-1}+2\alpha_{j}+\cdots+2\alpha_{r-1}+\alpha_{r}) \mid 1\leq i < j \leq l\}\\
&\cup
\{\pm(2\alpha_{i}+\cdots+2\alpha_{r-1}+\alpha_{r}) \mid 1\leq i \leq l\}.
\end{align*}
\end{lemma}
\begin{lemma}\label{realCIII}
In the case where $(R, \theta)$ is of type
CIII$(\RANK=r, \SRANK=l)$,
the set of all real restricted roots of $R$
coincides with
$\{\pm(\alpha_{2i-1}+2\alpha_{2i}+\cdots+2\alpha_{r-1}+\alpha_{r}) \mid 1 \leq i \leq l\}$.
\end{lemma}
\begin{proof}
First, we consider the case of $r=2l$.
From the Satake diagram of $(R, \theta)$
the Satake involution is trivial
and $\varPsi(R_{0})=\{\alpha_{2i-1} \mid 1 \leq i \leq l\}$.
Since $\tilde{\theta}$ leaves the root system $R$ invariant and 
any positive root has the form
$\alpha_{i}+\cdots+\alpha_{j-1},
\alpha_{i}+\cdots+\alpha_{j-1}+2\alpha_{j}+\cdots+\alpha_{2l-1}+\alpha_{2l}$
or
$2\alpha_{i}+\cdots+2\alpha_{2l-1}+\alpha_{2l}$,
we have 
\begin{equation*}
\alpha^{\tilde{\theta}}_{2i} =
\begin{cases}
\alpha_{2i-1}+\alpha_{2i}+\alpha_{2i+1} & (1 \leq i \leq l-1)\\
2\alpha_{2l-1}+\alpha_{2l}              & (i=l)
\end{cases}
\end{equation*}
Therefore,
by direct calculation we can explicitly determine
the set $\{\alpha \in R \mid \alpha=\alpha^{\tilde{\theta}}\}$
as in the assertion.
Next,
we consider the case of $r>2l$.
By the same argument as above we have
$\alpha^{\tilde{\theta}}_{2i-1}=-\alpha_{2i-1} (1 \leq i \leq l)$,
$\alpha^{\tilde{\theta}}_{2l+k}=-\alpha_{2l+k} (1 \leq k \leq r-2l)$
and
\begin{equation*}
\alpha^{\tilde{\theta}}_{2i}=
\begin{cases}
\alpha_{2i-1}+\alpha_{2i}+\alpha_{2i+1} & (1 \leq i \leq l-1),\\
\alpha_{2l-1}+\alpha_{2l}+2\alpha_{2l+1}+\cdots+2\alpha_{r-1}+\alpha_{r} & (i=l).
\end{cases}
\end{equation*}
Therefore,
by direct calculation we can explicitly determine
the set $\{\alpha \in R \mid \alpha=\alpha^{\tilde{\theta}}\}$
as in the assertion.
\end{proof}
\noindent
By imitating the proof of Lemma \ref{realCIII}
we have the following fact.
\begin{lemma}\label{realBCIII}
In the case where $(R, \theta)$ is of type
BCIII$(\RANK=r, \SRANK=l)$,
the set of all real restricted roots of $R$
coincides with
$\{\pm(\alpha_{2i-1}+2\alpha_{2i}+\cdots+2\alpha_{r}) \mid 1\leq i \leq l\}$.
\end{lemma}
\begin{lemma}\label{realDI-2}
In the case where $(R, \theta)$ is of type
DI$(\RANK=r, \SRANK=l)$ $(r > l)$,
the set of all real restricted roots of $R$
coincides with
\begin{equation*}
\{\pm(\alpha_{i}+\cdots+\alpha_{j-1}) \mid 1\leq i<j \leq l\}\\
\cup
\{\pm(\alpha_{i}+\cdots+\alpha_{r-2}+\alpha_{j}+\cdots+\alpha_{r}) \mid 1 \leq i<j \leq l\}.
\end{equation*}
\end{lemma}
\begin{proof}
First, we consider the case of $r=l+1$.
From the Satake diagram of $(R, \theta)$
we have $\varPsi(R_{0}) = \emptyset$ and 
\begin{equation*}
\alpha^{\tilde{\theta}}_{i}=p\alpha_{i}=
\begin{cases}
\alpha_{i} & (1\leq i \leq r-2),\\
\alpha_{r} & (i=r-1),\\
\alpha_{r-1} & (i=r).
\end{cases}
\end{equation*}
By direct calculation we have
\begin{align*}
(\alpha_{i}+\cdots+\alpha_{j-1})^{\tilde{\theta}}&=
\begin{cases}
\alpha_{i}+\cdots+\alpha_{j-1} & (1\leq i<j \leq r-1),\\
\alpha_{i}+\cdots+\alpha_{r-2}+\alpha_{r} & (1\leq i < j=r),
\end{cases}\\
(\alpha_{i}+\cdots+\alpha_{r-2}+\alpha_{j}+\cdots+\alpha_{r})^{\tilde{\theta}}
&\\
&\hspace{-0.1\textwidth}=\begin{cases}
\alpha_{i}+\cdots+\alpha_{r-2}+\alpha_{j}+\cdots+\alpha_{r} & (1\leq i< j \leq r-1),\\
\alpha_{i}+\cdots+\alpha_{r-1} & (1 \leq i<j =r).
\end{cases}
\end{align*}
This proves the statement.
Next, we consider  the case of $r>l+1$.
From the Satake diagram of $(R, \theta)$
the Satake involution is trivial and
$\varPsi(R_{0})=\{\alpha_{l+k} \mid 1\leq i\leq r-l\}$.
Since $\tilde{\theta}$ leaves the root system $R$ invariant,
we have
\begin{equation*}
\alpha^{\tilde{\theta}}_{i}=
\begin{cases}
\alpha_{i} & (1 \leq i \leq l-1),\\
\alpha_{l}+\cdots+\alpha_{r-2}+\alpha_{l+1}+\cdots+\alpha_{r} & (i=l).
\end{cases}
\end{equation*}
Therefore,
by direct calculation we can explicitly determine
the set $\{\alpha\in R \mid \alpha=\alpha^{\tilde{\theta}}\}$
as in the assertion.
\end{proof}
\noindent
By imitating the proof of Lemma \ref{realDI-2} we have the following fact.
\begin{lemma}\label{realDIII-1}
In the case where $(R, \theta)$ is of type
DIII$(\RANK=r, \SRANK=l)$,
the set of all real restricted roots of $R$
coincides with
\begin{equation*}
\{\pm(\alpha_{2i-1}+\cdots+\alpha_{r-2}+\alpha_{2i}+\cdots+\alpha_{r}) \mid 1\leq i \leq l\}.
\end{equation*}
\end{lemma}


\begin{lemma}\label{realEII}
In the case where $(R, \theta)$ is of type EII,
the set of all real restricted roots of $R$ coincides with
\begin{align*}
\{&\pm\alpha_{2}, \pm\alpha_{4}, \pm(\alpha_{3}+\alpha_{4}+\alpha_{5}),
\pm(\alpha_{2}+\alpha_{4}),
\pm(\alpha_{2}+\alpha_{3}+\alpha_{4}+\alpha_{5})\}\\
&\cup\{\pm(\alpha_{2}+\alpha_{3}+2\alpha_{4}+\alpha_{5}),
\pm(\alpha_{1}+\alpha_{3}+\alpha_{4}+\alpha_{5}+\alpha_{6})\}\\
&\cup\{\pm(\alpha_{1}+\alpha_{2}+\alpha_{3}+\alpha_{4}+\alpha_{5}+\alpha_{6})\}\\
&\cup\{\pm(\alpha_{1}+\alpha_{2}+\alpha_{3}+2\alpha_{4}+\alpha_{5}+\alpha_{6}),
\pm(\alpha_{1}+\alpha_{2}+2\alpha_{3}+2\alpha_{4}+2\alpha_{5}+\alpha_{6})\}\\
&\cup\{\pm(\alpha_{1}+\alpha_{2}+2\alpha_{3}+3\alpha_{4}+2\alpha_{5}+\alpha_{6}),
\pm(\alpha_{1}+2\alpha_{2}+2\alpha_{3}+3\alpha_{4}+2\alpha_{5}+\alpha_{6})\}.
\end{align*}
\end{lemma}
\begin{proof}
From the Satake diagram of $(R, \theta)$
the Satake involution $p$ satisfies
$
p\alpha_{1}=\alpha_{6},
p\alpha_{2}=\alpha_{2},
p\alpha_{3}=\alpha_{5}
$
and $p\alpha_{4}=\alpha_{4}$, and $\varPsi(R_{0})=\emptyset$.
Therefore we have
$
\alpha_{1}^{\tilde{\theta}} = \alpha_{6},
\alpha_{2}^{\tilde{\theta}} = \alpha_{2},
\alpha_{3}^{\tilde{\theta}} = \alpha_{5}
$
and $\alpha^{\tilde{\theta}}_{4} = \alpha_{4}$.
If we put $\alpha = \sum^{6}_{i=1} n_{i}\alpha_{i} \in R$ then,
$\alpha = \alpha^{\tilde{\theta}}$ holds if and only if $\alpha$ satisfies
$n_{1}=n_{6}$ and $n_{3}=n_{5}$.
Therefore, by direct calculation
we can explicitly determine the set
$\{\alpha \in R \mid \alpha = \alpha^{\tilde{\theta}}\}$
as in the assertion. 
\end{proof}
\begin{lemma}\label{realEIII}
In the case where $(R, \theta)$ is of type EIII,
the set of all real restricted roots of $R$ coincides with
$\{
\pm(\alpha_{1}+\alpha_{3}+\alpha_{4}+\alpha_{5}+\alpha_{6})
\pm(\alpha_{1}+2\alpha_{2}+2\alpha_{3}+3\alpha_{4}+2\alpha_{5}+\alpha_{6})
\}$.
\end{lemma}
\begin{proof}
From the Satake diagram of $(R, \theta)$
the Satake involution $p$ satisfies
$p\alpha_{1} = \alpha_{6}$ and $p\alpha_{2} = \alpha_{2}$,
and $\varPsi(R_{0})=\{\alpha_{3},\alpha_{4},\alpha_{5}\}$.
Therefore the possibility of the form $\alpha^{\tilde{\theta}}_{1}$
is either $\alpha_{6}, \alpha_{5}+\alpha_{6}, \alpha_{4}+\alpha_{5}+\alpha_{6}$
or $\alpha_{3}+\alpha_{4}+\alpha_{5}+\alpha_{6}$.
Since $\tilde{\theta}$ leaves the root system $R$ invariant,
we have $\alpha^{\tilde{\theta}}_{1} = \alpha_{3}+\alpha_{4}+\alpha_{5}+\alpha_{6}$.
The same argument shows $\alpha^{\tilde{\theta}}_{2} = \alpha_{2}+\alpha_{3}+2\alpha_{4}+\alpha_{5}$.
Since $\tilde{\theta}$ is involutive,
we have
$\alpha^{\tilde{\theta}}_{6} = \alpha_{1}+\alpha_{3}+\alpha_{4}+\alpha_{5}$.
If we put $\alpha = \sum^{6}_{i=1}n_{i}\alpha_{i} \in R$ then,
$\alpha=\alpha^{\tilde{\theta}}$ holds
if and only if
$\alpha$ satisfies
$n_{1}=n_{6}, n_{1}+n_{2}+n_{6}=2n_{3}, n_{1}+2n_{2}+n_{6}=2n_{4}$
and $n_{1}+n_{2}+n_{6}=2n_{5}$.
Therefore, by direct calculation
we can explicitly determine the set
$\{\alpha \in R \mid \alpha = \alpha^{\tilde{\theta}}\}$
as in the assertion.
\end{proof}
\begin{lemma}\label{realEIV}
In the case where $(R, \theta)$ is of type EIV,
there exists no real restricted root in $R$.
\end{lemma}
\begin{proof}
From the Satake diagram of $(R, \theta)$
the Satake involution is trivial
and $\varPsi(R_{0}) = \{\alpha_{2},\alpha_{3},\alpha_{4},\alpha_{5}\}$.
Therefore the possibility of the form $\alpha^{\tilde{\theta}}_{1}$
is either
$
\alpha_{1},
\alpha_{1}+\alpha_{3},
\alpha_{1}+\alpha_{3}+\alpha_{4},
\alpha_{1}+\alpha_{2}+\alpha_{3}+\alpha_{4},
\alpha_{1}+\alpha_{3}+\alpha_{4}+\alpha_{5},
\alpha_{1}+\alpha_{2}+\alpha_{3}+\alpha_{4}+\alpha_{5},
\alpha_{1}+\alpha_{2}+\alpha_{3}+2\alpha_{4}+\alpha_{5}
$
or $\alpha_{1}+\alpha_{2}+2\alpha_{3}+2\alpha_{4}+\alpha_{5}$.
Since $\tilde{\theta}$ leaves the root system $R$ invariant,
we have $\alpha^{\tilde{\theta}}_{1} = \alpha_{1}+\alpha_{2}+2\alpha_{3}+2\alpha_{4}+\alpha_{5}$.
The same argument shows $\alpha^{\tilde{\theta}}_{6} = \alpha_{2}+\alpha_{3}+2\alpha_{4}+2\alpha_{5}+\alpha_{6}$.
If we put $\alpha = \sum^{6}_{i=1}n_{i}\alpha_{i} \in R$ then,
$\alpha = \alpha^{\tilde{\theta}}$ holds if and only if $\alpha$ satisfies
$
2n_{2} = n_{1}+n_{6},
2n_{3}=2n_{1}+n_{6},
n_{4}=n_{1}+n_{6}
$
and $2n_{5}=n_{1}+2n_{6}$.
Therefore, by direct calculation
we have 
$\{\alpha \in R \mid \alpha = \alpha^{\tilde{\theta}}\}=\emptyset$.
\end{proof}
\noindent
By imitating the proof of Lemma \ref{realEIV},
we have the following five facts.
\begin{lemma}\label{realEVI}
In the case where $(R, \theta)$ is of type EVI,
the set of all real restricted roots of $R$ coincides with
\begin{align*}
\{
&\pm \alpha_{1},
\pm \alpha_{3},
\pm(\alpha_{1}+\alpha_{3}),
\pm(\alpha_{2}+\alpha_{3}+2\alpha_{4}+\alpha_{5}),
\pm(\alpha_{1}+\alpha_{2}+\alpha_{3}+2\alpha_{4}+\alpha_{5})
\}\\
&\cup\{
\pm(\alpha_{1}+\alpha_{2}+2\alpha_{3}+2\alpha_{4}+\alpha_{5}),
\pm(\alpha_{1}+\alpha_{2}+2\alpha_{3}+2\alpha_{4}+2\alpha_{5}+2\alpha_{6}+\alpha_{7})
\}\\
&\cup\{
\pm(\alpha_{1}+\alpha_{2}+\alpha_{3}+2\alpha_{4}+2\alpha_{5}+2\alpha_{6}+\alpha_{7})\}\\
&\cup\{\pm(2\alpha_{1}+2\alpha_{2}+3\alpha_{3}+4\alpha_{4}+3\alpha_{5}+2\alpha_{6}+\alpha_{7})
\}\\
&\cup\{
\pm(\alpha_{1}+2\alpha_{2}+2\alpha_{3}+4\alpha_{4}+3\alpha_{5}+2\alpha_{6}+\alpha_{7})\}\\
&\cup\{\pm(\alpha_{1}+2\alpha_{2}+3\alpha_{3}+4\alpha_{4}+3\alpha_{5}+2\alpha_{6}+\alpha_{7})
\}\\
&\cup\{
\pm(\alpha_{2}+\alpha_{3}+2\alpha_{4}+2\alpha_{5}+2\alpha_{6}+\alpha_{7})
\}.
\end{align*}
\end{lemma}
\begin{lemma}\label{realEVII}
In the case where $(R, \theta)$ is of type EVII,
the set of all real restricted roots of $R$ coincides with
\begin{equation*}
\{
\pm \alpha_{7},
\pm(\alpha_{2}+\alpha_{3}+2\alpha_{4}+2\alpha_{5}+2\alpha_{6}+\alpha_{7}),
\pm(2\alpha_{1}+2\alpha_{2}+3\alpha_{3}+4\alpha_{4}+3\alpha_{5}+2\alpha_{6}+\alpha_{7})
\}.
\end{equation*}
\end{lemma}
\begin{lemma}\label{realEIX}
In the case where $(R, \theta)$ is of type EIX,
the set of all real restricted roots of $R$
coincides with
\begin{align*}
\{&\pm \alpha_{7},\pm \alpha_{8},\pm(\alpha_{7}+\alpha_{8}),\pm(\alpha_{2}+\alpha_{3}+2\alpha_{4}+2\alpha_{5}+2\alpha_{6}+\alpha_{7})\}\\
&\phantom{}\cup\{\pm(\alpha_{2}+\alpha_{3}+2\alpha_{4}+2\alpha_{5}+2\alpha_{6}+\alpha_{7}+\alpha_{8})\}\\
&\phantom{}\cup\{\pm(\alpha_{2}+\alpha_{3}+2\alpha_{4}+2\alpha_{5}+2\alpha_{6}+2\alpha_{7}+\alpha_{8})\}\\
&\phantom{}\cup\{\pm(2\alpha_{1}+2\alpha_{2}+3\alpha_{3}+4\alpha_{4}+3\alpha_{5}+2\alpha_{6}+\alpha_{7})\}\\
&\phantom{}\cup\{\pm(2\alpha_{1}+2\alpha_{2}+3\alpha_{3}+4\alpha_{4}+3\alpha_{5}+2\alpha_{6}+2\alpha_{7}+\alpha_{8})\}\\
&\phantom{}\cup\{\pm(2\alpha_{1}+2\alpha_{2}+3\alpha_{3}+4\alpha_{4}+3\alpha_{5}+2\alpha_{6}+\alpha_{7}+\alpha_{8})\}\\
&\phantom{}\cup\{\pm(2\alpha_{1}+3\alpha_{2}+4\alpha_{3}+6\alpha_{4}+5\alpha_{5}+4\alpha_{6}+2\alpha_{7}+\alpha_{8})\}\\
&\phantom{}\cup\{\pm(2\alpha_{1}+3\alpha_{2}+4\alpha_{3}+6\alpha_{4}+5\alpha_{5}+4\alpha_{6}+3\alpha_{7}+\alpha_{8})\}\\
&\phantom{}\cup\{\pm(2\alpha_{1}+3\alpha_{2}+4\alpha_{3}+6\alpha_{4}+5\alpha_{5}+4\alpha_{6}+3\alpha_{7}+2\alpha_{8})\}.
\end{align*}
\end{lemma}
\begin{lemma}\label{realFII}
In the case where $(R, \theta)$ is of type FII,
the set of all real restricted roots of $R$
coincides with $\{
\pm(\alpha_{1}+2\alpha_{2}+3\alpha_{3}+2\alpha_{4})
\}$.
\end{lemma}
\begin{lemma}\label{realFIII}
In the case where $(R, \theta)$ is of type FIII,
the set of all real restricted roots of $R$
coincides with
\begin{equation*}
\{
\pm(\alpha_{1}+\alpha_{2}+\alpha_{3}),
\pm(\alpha_{2}+2\alpha_{3}+2\alpha_{4}),
\pm(\alpha_{1}+2\alpha_{2}+3\alpha_{3}+2\alpha_{4}),
\pm(2\alpha_{1}+3\alpha_{2}+4\alpha_{3}+2\alpha_{4})
\}.
\end{equation*}
\end{lemma}
\begin{lemma}\label{realanti}
In the case where $(R, \tilde{\theta})$ is of type
A$+$A, B$+$B, C$+$C, D$+$D, BC$+$BC, EI$+$EI,
EV$+$EV, EVIII$+$EVIII, FI$+$FI or G$+$G,
there exists no real restricted root in $R$.
\end{lemma}
\begin{proof}
The restricted root system $R$ has two irreducible components
$R^{1}, R^{2}$,
which are isomorphic to each other.
Set $\varPsi(R^{j}) = \{\alpha^{j}_{1}, \ldots, \alpha^{j}_{r}\} (r=\RANK R^{j}, j=1,2)$.
Renumbering $\alpha^{j}_{i}$,
if necessary,
we assume that $\alpha^{j}_{1} > \cdots > \alpha^{j}_{r}$.
From the Satake diagram of $(R, \tilde{\theta})$
we have 
$\varPsi(R^{j}_{0})=\emptyset$ and
$p \alpha^{2}_{i} = \alpha^{1}_{i} (1 \leq i \leq r)$.
This implies that $(\alpha^{2}_{i})^{\tilde{\theta}} = \alpha^{1}_{i}
(1\leq i \leq r)$.
Since any restricted root in $R$ is a linear combination of either
$\{\alpha^{1}_{1}, \ldots, \alpha^{1}_{r}\}$ or
$\{\alpha^{2}_{1}, \ldots, \alpha^{2}_{r}\}$,
there exists no real restricted root.
\end{proof}

By using Corollary \ref{cor.real} and Theorem \ref{thm.realroots}
we have Theorem stated in Introduction.

\begin{remark}
By imitating our method we can give examples
of austere orbits in a pseudo-hyperbolic space
$\BS{H}$ ($:=\{ v \in \MF{q} \mid  B(v,v)=r (<0)\}$).
In fact,
for any imaginary root $\alpha$,
the orbit through $\sqrt{-1}A_{\alpha}$
is an austere orbit in $\BS{H}$.
Moreover,
the Dynkin diagram of
the subsystem $\{\alpha \in R \mid \theta(\alpha)=\alpha\}$
can be determined by the black circles in the Satake
diagram associated with $(R, \theta)$.
\end{remark}

\appendix
\section{Satake diagram of $(R, \theta)$}\label{sec.appendix}

Let $(\MF{g}, \MF{h})$
be a semisimple symmetric pair,
$\sigma$ be an involution of $\MF{g}$ with $\OPE{Ker}(\sigma - \OPE{id})=\MF{h}$,
and $\theta$ be a Cartan involution $\theta$ commuting with $\sigma$.
Denote by $R$ the restricted root system of
$(\MF{g}, \MF{h})$ with respect to a $\theta$-invariant MSCS $\MF{a}$ of $\MF{q}$.
In this appendix,
we determine the Satake diagram of $(R, \theta, \MF{a})$.
Set $\MF{g}^{d}=\MF{k}\cap\MF{h} + \sqrt{-1}(\MF{p}\cap\MF{h}) + \sqrt{-1}(\MF{k}\cap\MF{q})
+ \MF{p}\cap\MF{q}$,
which is a subalgebra of $\MF{g}^{\BS{C}}$.
The involutions $\sigma$ and $\theta$
induce involutions of $\MF{g}^{d}$,
which are also denoted by the same symbol $\sigma$ and $\theta$,
respectively.
In particular, $\sigma$ is a Cartan involution of $\MF{g}^{d}$.
Denoted by $\MF{k}^{d}$ (resp.\ $\MF{p}^{d}$)
the $(+1)$-eigenspace (resp.\ the $(-1)$-eigenspace)
of $\sigma$ in $\MF{g}^{d}$.
Then we have $\MF{a}_{\BS{R}}$ ($:=\sqrt{-1}(\MF{k}\cap\MF{a})+\MF{p}\cap\MF{a}$)
is a maximal abelian subspace of $\MF{p}^{d}$ ($=\sqrt{-1}(\MF{k}\cap\MF{q})+\MF{p}\cap\MF{q}$).
Note that $R$ give also the restricted root system of the Riemannian symmetric
pair $(\MF{g}^{d}, \MF{k}^{d})$ with respect to $\MF{a}_{\BS{R}}$.
Let $\MF{a}_{\MF{p}}$
be a maximal abelian subspace of $\MF{p}$ containing $\MF{p}\cap\MF{a}$.
Since $\MF{p}\cap\MF{a}$ is maximal in $\MF{p}\cap\MF{q}$,
we have $[\MF{a}, \MF{a}_{\MF{p}}]=\{0\}$ (cf.\ \cite[Lemma 2.4]{MR810638}).
If $\tilde{\MF{a}}$ is a maximal abelian subalgebra of $\MF{g}$
containing $\MF{a}$ and $\MF{a}_{\MF{p}}$,
then $\tilde{\MF{a}}$ is a Cartan subalgebra of $\MF{g}$.
Denote by $\Sigma$ the root system of $\MF{g}^{\BS{C}}$
with respect to $\tilde{\MF{a}}^{\BS{C}}$.
We can give a $(\theta, \sigma)$-fundamental system $\varPsi$ of $\Sigma$
(cf.\ \cite{MR810638} for the definition of a $(\theta, \sigma)$-fundamental system).
Therefore,
$\varPsi$ gives the Satake diagram of Riemannian symmetric pairs
$(\MF{g}, \MF{k})$ and $(\MF{g}^{d}, \MF{k}^{d})$,
which are denoted by $S(\MF{g},\MF{k})$ and $S(\MF{g}^{d}, \MF{k}^{d})$,
respectively.
Then we give a recipe to determine the Satake diagram of $(R, \theta, \MF{a})$
by using $S(\MF{g},\MF{k})$ and $S(\MF{g}^{d}, \MF{k}^{d})$
as follows.

\begin{recipe}
\label{recipe.satake}
Denote by $S(R, \theta, \MF{a})$ the Satake diagram of $(R, \theta, \MF{a})$.
\begin{enumerate}[(Step 1)]
\item For each $\alpha \in \varPsi$,
we determine $\sigma(\alpha)$ (resp.\ $\theta(\alpha)$)
by using $S(\MF{g},\MF{k})$ (resp.\ $S(\MF{g}^{d}, \MF{k}^{d})$).
\item We give the set
$\{\alpha \in \varPsi \mid \alpha|_{\MF{a}^{\BS{C}}} = 0\}$ ($=:\varPsi_{0}$), and
determine
the Dynkin diagram of $R$
by investigating $\{\alpha|_{\MF{a}^{\BS{C}}} \mid \alpha \in \varPsi\setminus\varPsi_{0}\}$ ($=:\overline{\varPsi}$).
In fact, we calculate $(\alpha - \sigma(\alpha))/2$ as $\alpha|_{\MF{a}^{\BS{C}}}$ for each $\alpha \in \varPsi$.
\item We determine $\{\lambda \in \overline{\varPsi} \mid \lambda|_{\MF{p}\cap\MF{a}}=0\}$ ($=:\overline{\varPsi}_{0}$)
by investigating $\{\alpha \in \varPsi \mid \alpha|_{\MF{a}_{\MF{p}}}=0\}$.
In fact, for each $\lambda = (\alpha - \sigma(\alpha))/2\ (\alpha \in \varPsi\setminus\varPsi_{0})$,
we determine whether or not $\alpha|_{\MF{a}_{\MF{p}}} = 0$ holds,
that is, $\alpha$ is a black circle in $S(\MF{g}, \MF{k})$.
Then the elements in $\overline{\varPsi}_{0}$ are black circles in $S(R, \theta, \MF{a})$.
\item For any
$\lambda_{1}, \lambda_{2} \in \overline{\varPsi}\setminus\overline{\varPsi}_{0}, \lambda_{1}\neq\lambda_{2}$,
we determine whether or not $\lambda_{1}|_{\MF{p}\cap\MF{a}}=\lambda_{2}|_{\MF{p}\cap\MF{a}}$ holds
by calculating $\lambda_{i}-\theta(\lambda_{i}) (i=1,2)$.
In fact, we calculate $(\alpha - \sigma(\alpha) - \theta(\alpha) + \sigma(\theta(\alpha))/4$
as $\alpha|_{\MF{a}^{\BS{C}}} - \theta(\alpha|_{\MF{a}^{\BS{C}}}$) ($\alpha \in \varPsi\setminus\varPsi_{0}$). 
If $\lambda_{1}|_{\MF{p}\cap\MF{a}}=\lambda_{2}|_{\MF{p}\cap\MF{a}}$ holds,
then $\lambda_{1}|_{\MF{p}\cap\MF{a}}, \lambda_{2}|_{\MF{p}\cap\MF{a}}$ are joined with $\leftrightarrow$.
\end{enumerate}
\end{recipe}

\noindent
By using Recipe \ref{recipe.satake} we shall list up the Satake diagrams of $(R,\theta)$
associated with MSCSs for all irreducible
pseudo-Riemannian symmetric pairs.
In Table \ref{table.satakeclassical}, we give the list of the irreducible pseudo-Riemannian
symmetric pairs and their types of $(R, \theta)$ associated with MSCSs.
In Table \ref{table.satakelistclassical}, we describe the Satake diagrams
of $(R,\theta)$.


\begin{table}[!!h]
\scriptsize
\renewcommand{\arraystretch}{1.1}
\caption{The Type of $(R, \theta)$}\label{table.satakeclassical}
\centering
\begin{minipage}{0.8\textwidth}
\small
(i-a) $\MF{g}$ is classical and $\MF{g}$ is noncompact simple with no complex structure.
\end{minipage}

\begin{tabular}{|c|c|c|c|}
\multicolumn{4}{l}{
$(\MF{g},\MF{h})=(\MF{su}(n,m),\MF{su}(i,j)+\MF{su}(n-i,m-j)+\MF{so}(2))$}\\
\hline
Type of $(R,\theta)$ & $\RANK$ & $\SRANK$ & Remarks \\
\hline
\hline
CI      & \multirow{2}{*}{$\min(i+j, m+n-(i+j))$}     & \multirow{2}{*}{$\min(i,m-j)+\min(j,n-i)$}      & $m+n = 2(i+j)$ \\\cline{1-1}\cline{4-4}
BI      &  &  & $m+n\neq 2(i+j)$\\\hline
\end{tabular}

\bigskip

\begin{tabular}{|c|c|c|c|}
\multicolumn{4}{l}{
$(\MF{g},\MF{h})=(\MF{so}(n,m),\MF{so}(i,j)+\MF{so}(n-i,m-j))$}\\
\hline
Type of $(R,\theta)$ & $\RANK$ & $\SRANK$ & Remarks \\
\hline
\hline
DI      & \multirow{2}{*}{$\min(i+j, m+n-(i+j))$}     & \multirow{2}{*}{$\min(i,m-j)+\min(j,n-i)$}      & $m+n = 2(i+j)$ \\\cline{1-1}\cline{4-4}
BI      &  &  & $m+n\neq 2(i+j)$\\\hline
\end{tabular}

\bigskip

\begin{tabular}{|c|c|c|c|}
\multicolumn{4}{l}{
$(\MF{g},\MF{h})=(\MF{sp}(n,m),\MF{sp}(i,j)+\MF{sp}(n-i,m-j))$}\\
\hline
Type of $(R,\theta)$ & $\RANK$ & $\SRANK$ & Remarks \\
\hline
\hline
CI      & \multirow{2}{*}{$\min(i+j, m+n-(i+j))$}     & \multirow{2}{*}{$\min(i,m-j)+\min(j,n-i)$}      & $m+n = 2(i+j)$ \\\cline{1-1}\cline{4-4}
BCI      &  &  & $m+n\neq 2(i+j)$\\\hline
\end{tabular}

\bigskip
\bigskip

\begin{tabular}{|l|c|c|c|c|} 
\hline
Symmetric pair $(\MF{g},\MF{h})$ & Type of $(R, \theta)$ & $\RANK$ & $\SRANK$ & Remarks\\
\hline
\hline
\LISTV{$(\MF{sl}(n,\BS{R}),\MF{so}(p,n-p))$}{AI}{$n-1$}{$n-1$}{}
\LISTV{$(\MF{su}(p,n-p),\MF{so}(p,n-p))$}{AIII}{$n-1$}{$\min(p,n-p)$}{}
\multirow{2}{*}{$(\MF{sl}(n,\BS{R}),\MF{sl}(p,\BS{R})+\MF{sl}(n-p,\BS{R})+\BS{R})$}
& CI & \multirow{2}{*}{$p$} & \multirow{2}{*}{$p$} & $n=2p$ \\\cline{2-2}\cline{5-5}
& BCI   &  &  & $n>2p$ \\\hline
\LISTV{$(\MF{sl}(2n,\BS{R}),\MF{sp}(n,\BS{R}))$}{AI}{$n-1$}{$n-1$}{}
\LISTV{$(\MF{su}^{*}(2n),\MF{so}^{*}(2n))$}{AII}{$2n-1$}{$n-1$}{}
\LISTV{$(\MF{su}(n,n),\MF{so}^{*}(2n))$}{AIII}{$2n-1$}{$n$}{}
\LISTV{$(\MF{sl}(2n,\BS{R}),\MF{sl}(n,\BS{C})+\MF{so}(2))$}{CI}{$n$}{$n$}{}
\LISTV{$(\MF{su}^{*}(2n),\MF{sl}(n,\BS{C})+\MF{so}(2))$}{CIII}{$n$}{$[n/2]$}{}
\LISTV{$(\MF{su}(n,n),\MF{sp}(n,\BS{R}))$}{AIII}{$n-1$}{$[n/2]$}{}
\LISTV{$(\MF{su}(n,n),\MF{sl}(n,\BS{C})+\BS{R})$}{CI}{$n$}{$n$}{}
\LISTV{$(\MF{su}^{*}(2n),\MF{sp}(p,n-p))$}{AI}{$n-1$}{$n-1$}{}
\LISTV{$(\MF{su}(2p,2(n-p)),\MF{sp}(p,n-p))$}{AIII}{$n-1$}{$\min(p,n-p)$}{}
\multirow{2}{*}{$(\MF{su}^{*}(2n),\MF{su}^{*}(2p)+\MF{su}^{*}(2(n-p))+\BS{R})$}
& CIII & \multirow{2}{*}{$2p$} & \multirow{2}{*}{$p$} & $n=2p$ \\\cline{2-2}\cline{5-5}
& BCIII   &  &  & $n>2p$ \\\hline
\multirow{2}{*}{$(\MF{so}^{*}(2n),\MF{su}(p,n-p)+\MF{so}(2))$}
& CI & \multirow{2}{*}{$[n/2]$} & \multirow{2}{*}{$[n/2]$} & $n$: even \\\cline{2-2}\cline{5-5}
& BCI   &  &  & $n$: odd \\\hline
\multirow{2}{*}{$(\MF{so}(2p,2(n-p)),\MF{su}(p,n-p)+\MF{so}(2))$}
& CI & \multirow{2}{*}{$[n/2]$} & \multirow{2}{*}{$\min(p,n-p)$} & $n$: even \\\cline{2-2}\cline{5-5}
& BCI   &  &  & $n$: odd \\\hline
\multirow{2}{*}{$(\MF{so}^{*}(2n),\MF{so}^{*}(2p)+\MF{so}^{*}(2(n-p)))$}
& DIII & \multirow{2}{*}{$2p$} & \multirow{2}{*}{$p$} & $n=2p$ \\\cline{2-2}\cline{5-5}
& BI  &  &  & $n>2p$ \\\hline
\LISTV{$(\MF{so}(n,n),\MF{so}(n,\BS{C}))$}{DI}{$n$}{$n$}{}
\LISTV{$(\MF{so}^{*}(2n),\MF{so}(n,\BS{C}))$}{DIII}{$n$}{$[n/2]$}{}

\multirow{2}{*}{$(\MF{so}(n,n),\MF{sl}(n,\BS{R})+\BS{R})$}
& CI & \multirow{2}{*}{$[n/2]$} & \multirow{2}{*}{$[n/2]$} & $n$: even \\\cline{2-2}\cline{5-5}
& BCI  &  &  & $n$: odd \\\hline
\end{tabular}
\end{table}

\begin{table}[!!h]
\scriptsize
\renewcommand{\arraystretch}{1.1}
\contcaption{(continued)}
\centering
\begin{tabular}{|l|c|c|c|c|} 
\hline
Symmetric pair $(\MF{g},\MF{h})$ & Type of $(R, \theta)$ & $\RANK$ & $\SRANK$ & Remarks\\
\hline
\hline
\LISTV{$(\MF{so}^{*}(4n),\MF{su}^{*}(2n)+\BS{R})$}{CI}{$n$}{$n$}{}
\LISTV{$(\MF{sp}(n,\BS{R}),\MF{su}(p,n-p)+\MF{so}(2))$}{CI}{$n$}{$n$}{}
\LISTV{$(\MF{sp}(p,n-p),\MF{su}(p,n-p)+\MF{so}(2))$}{CIII}{$n$}{$\min(p,n-p)$}{}
\multirow{2}{*}{$(\MF{sp}(n,\BS{R}),\MF{sp}(p,\BS{R})+\MF{sp}(n-p,\BS{R}))$}
& CI & \multirow{2}{*}{$p$} & \multirow{2}{*}{$p$} & $n=2p$ \\\cline{2-2}\cline{5-5}
& BI &  &  & $n>2p$ \\\hline
\LISTV{$(\MF{sp}(n,\BS{R}),\MF{sl}(n,\BS{R})+\BS{R})$}{CI}{$n$}{$n$}{}
\LISTV{$(\MF{sp}(n,n),\MF{sp}(n,\BS{C}))$}{CI}{$n$}{$n$}{}
\LISTV{$(\MF{sp}(2n,\BS{R}),\MF{sp}(n,\BS{C}))$}{CI}{$n$}{$n$}{}
\LISTV{$(\MF{sp}(n,n),\MF{su}^{*}(2n)+\BS{R})$}{CIII}{$2n$}{$n$}{}
\end{tabular}

\bigskip
\bigskip

\begin{minipage}{.8\textwidth}
\small
(i-b) $\MF{g}$ is classical and $\MF{g}$ is simple with a complex structure or the direct sum of two noncompact simple Lie algebras with no complex structure.
\end{minipage}
\begin{tabular}{|l|c|c|c|c|}
\hline
Symmetric pair $(\MF{g},\MF{h})$ & Type of $(R, \theta)$ & $\RANK$ & $\SRANK$ & Remarks\\
\hline
\hline
\LISTV{$(\MF{sl}(n,\BS{C}),\MF{sl}(n,\BS{R}))$}{AIII}{$n-1$}{$[n/2]$}{}
\LISTV{$(\MF{sl}(n,\BS{R})^{2},\MF{sl}(n,\BS{R}))$}{AI}{$n-1$}{$n-1$}{}
\LISTV{$(\MF{sl}(n,\BS{C}),\MF{so}(n,\BS{C}))$}{A$+$A}{$2(n-1)$}{$n-1$}{}
\LISTV{$(\MF{sl}(2n,\BS{C}),\MF{su}^{*}(2n))$}{AI}{$2n-1$}{$n$}{}
\LISTV{$(\MF{su}^{*}(2n)^{2},\MF{su}^{*}(2n))$}{AII}{$2n-1$}{$n-1$}{}
\LISTV{$(\MF{sl}(2n,\BS{C}),\MF{sp}(n,\BS{C}))$}{A$+$A}{$2(n-1)$}{$n-1$}{}
\LISTV{$(\MF{sl}(n,\BS{C}),\MF{su}(p,n-p))$}{AI}{$n-1$}{$n-1$}{}
\LISTV{$(\MF{su}(p,n-p)^{2},\MF{su}(p,n-p))$}{AIII}{$n-1$}{$\min(p,n-p)$}{}
\multirow{2}{*}{$(\MF{sl}(n,\BS{C}),\MF{sl}(p,\BS{C})+\MF{sl}(n-p,\BS{C})+\BS{C})$}
& C$+$C & \multirow{2}{*}{$2p$} & \multirow{2}{*}{$p$} & $n=2p$ \\\cline{2-2}\cline{5-5}
& BC$+$BC   &  &  & $n>2p$ \\\hline
\LISTV{$(\MF{so}(2n,\BS{C}),\MF{so}^{*}(2n))$}{DI}{$n$}{$n$}{}
\LISTV{$(\MF{so}^{*}(2n)^{2},\MF{so}^{*}(2n))$}{DIII}{$n$}{$[n/2]$}{}
\multirow{2}{*}{$(\MF{so}(2n,\BS{C}),\MF{sl}(n,\BS{C})+\BS{C})$}
& C$+$C & \multirow{2}{*}{$2[n/2]$} & \multirow{2}{*}{$[n/2]$} & $n$: even \\\cline{2-2}\cline{5-5}
& BC$+$BC   &  &  & $n$: odd \\\hline
\multirow{2}{*}{$(\MF{so}(n,\BS{C}),\MF{so}(p,n-p))$}
& DI & \multirow{2}{*}{$[n/2]$} & \multirow{2}{*}{$[n/2]$} & $n$: even \\\cline{2-2}\cline{5-5}
& BI   &  &  & $n$: odd \\\hline
\multirow{2}{*}{$(\MF{so}(p,n-p)^{2},\MF{so}(p,n-p))$}
& DI & \multirow{2}{*}{$[n/2]$} & \multirow{2}{*}{$\min(p,n-p)$} & $n$: even \\\cline{2-2}\cline{5-5}
& BI   &  &  & $n$: odd \\\hline
\multirow{2}{*}{$(\MF{so}(n,\BS{C}),\MF{so}(p,\BS{C})+\MF{so}(n-p,\BS{C}))$}
& D$+$D & \multirow{2}{*}{$2p$} & \multirow{2}{*}{$p$} & $n=2p$ \\\cline{2-2}\cline{5-5}
& B$+$B   &  &  & $n>2p$ \\\hline
\LISTV{$(\MF{sp}(n,\BS{C}),\MF{sp}(n,\BS{R}))$}{CI}{$n$}{$n$}{}
\LISTV{$(\MF{sp}(n,\BS{R})^{2},\MF{sp}(n,\BS{R}))$}{CI}{$n$}{$n$}{}
\LISTV{$(\MF{sp}(n,\BS{C}),\MF{sl}(n,\BS{C})+\BS{C})$}{C$+$C}{$2n$}{$n$}{}
\LISTV{$(\MF{sp}(n,\BS{C}),\MF{sp}(p,n-p))$}{CI}{$n$}{$n$}{}
\LISTV{$(\MF{sp}(p,n-p)^{2},\MF{sp}(p,n-p))$}{CIII}{$n$}{$\min(p,n-p)$}{}
\multirow{2}{*}{$(\MF{sp}(n,\BS{C}),\MF{sp}(p,\BS{C})+\MF{sp}(n-p,\BS{C}))$}
& C$+$C & \multirow{2}{*}{$2p$} & \multirow{2}{*}{$p$} & $n=2p$ \\\cline{2-2}\cline{5-5}
& BC$+$BC &  &  & $n>2p$ \\\hline
\end{tabular}
\end{table}


\begin{landscape}
\begin{table}[!!h]
\tiny
\renewcommand{\arraystretch}{1.1}
\contcaption{(continued)}
\centering
\begin{minipage}{0.9\textheight}
\small
(ii-a) $\MF{g}$ is exceptional and $\MF{g}$ is noncompact simple with no complex structure.
\end{minipage}
\setlength{\columnsep}{1pt}
\begin{multicols}{3}
\begin{tabular}{|l|c|c|c|}
\hline
\multirow{2}{*}{Symmetric pair $(\MF{g},\MF{h})$} & {\tiny Type of} & \multirow{2}{*}{$\RANK$} & \multirow{2}{*}{$\SRANK$} \\
&$(R,\theta)$ & & \\
\hline
\hline
\LISTIV{$(\MF{e}_{6(6)},\MF{sp}(4))$}{EI}{$6$}{$6$}
\LISTIV{$(\MF{e}_{6(6)},\MF{sp}(4,\BS{R}))$}{EI}{$6$}{$6$}
\LISTIV{$(\MF{e}_{6(6)},\MF{sl}(6,\BS{R})+\MF{sl}(2,\BS{R}))$}{FI}{$4$}{$4$}
\LISTIV{$(\MF{e}_{6(2)},\MF{sp}(4,\BS{R}))$}{EII}{$6$}{$4$}
\LISTIV{$(\MF{e}_{6(6)},\MF{sp}(2,2))$}{EI}{$6$}{$6$}
\LISTIV{$(\MF{e}_{6(6)},\MF{so}(5,5)+\BS{R})$}{BCI}{$2$}{$2$}
\LISTIV{$(\MF{e}_{6(-14)},\MF{sp}(2,2))$}{EIII}{$6$}{$2$}
\LISTIV{$(\MF{e}_{6(2)},\MF{su}(6)+\MF{su}(2))$}{FI}{$4$}{$4$}
\LISTIV{$(\MF{e}_{6(2)},\MF{su}(3,3)+\MF{sl}(2,\BS{R}))$}{FI}{$4$}{$4$}
\LISTIV{$(\MF{e}_{6(2)},\MF{su}(4,2)+\MF{su}(2))$}{FI}{$4$}{$4$}
\LISTIV{$(\MF{e}_{6(2)},\MF{so}(6,4)+\MF{so}(2))$}{BCI}{$2$}{$2$}
\LISTIV{$(\MF{e}_{6(-14)},\MF{su}(4,2)+\MF{su}(2))$}{FIII}{$4$}{$2$}
\LISTIV{$(\MF{e}_{6(-14)},\MF{so}(10)+\MF{so}(2))$}{BCI}{$2$}{$2$}
\LISTIV{$(\MF{e}_{6(-14)},\MF{so}^{*}(10)+\MF{so}(2))$}{BCI}{$2$}{$2$}
\LISTIV{$(\MF{e}_{6(-14)},\MF{su}(5,1)+\MF{sl}(2,\BS{R}))$}{FIII}{$4$}{$2$}
\LISTIV{$(\MF{e}_{6(2)},\MF{so}^{*}(10)+\MF{so}(2))$}{BCI}{$2$}{$2$}
\LISTIV{$(\MF{e}_{6(-14)},\MF{so}(8,2)+\MF{so}(2))$}{BCI}{$2$}{$2$}
\LISTIV{$(\MF{e}_{6(6)},\MF{f}_{4(4)})$}{AI}{$2$}{$2$}
\LISTIV{$(\MF{e}_{6(-26)},\MF{sp}(3,1))$}{EIV}{$6$}{$2$}
\LISTIV{$(\MF{e}_{6(6)},\MF{su}^{*}(6)+\MF{su}(2))$}{FI}{$4$}{$4$}
\LISTIV{$(\MF{e}_{6(2)},\MF{sp}(3,1))$}{EII}{$6$}{$4$}
\LISTIV{$(\MF{e}_{6(2)},\MF{f}_{4(4)})$}{AIII}{$2$}{$1$}
\LISTIV{$(\MF{e}_{6(-26)},\MF{su}^{*}(6)+\MF{su}(2))$}{FII}{$6$}{$1$}
\end{tabular}

\begin{tabular}{|l|c|c|c|}
\hline
\multirow{2}{*}{Symmetric pair $(\MF{g},\MF{h})$} & Type of & \multirow{2}{*}{$\RANK$} & \multirow{2}{*}{$\SRANK$} \\
&$(R,\theta)$ & & \\
\hline
\hline
\LISTIV{$(\MF{e}_{6(-26)},\MF{f}_{4})$}{AI}{$2$}{$2$}
\LISTIV{$(\MF{e}_{6(-26)},\MF{f}_{4(-20)})$}{AI}{$2$}{$2$}
\LISTIV{$(\MF{e}_{6(-26)},\MF{so}(9,1)+\BS{R})$}{BCIII}{$2$}{$1$}
\LISTIV{$(\MF{e}_{6(-14)},\MF{f}_{4(-20)})$}{AIII}{$2$}{$1$}
\LISTIV{$(\MF{e}_{7(7)},\MF{su}(8))$}{EV}{$7$}{$7$}
\LISTIV{$(\MF{e}_{7(7)},\MF{sl}(8,\BS{R}))$}{EV}{$7$}{$7$}
\LISTIV{$(\MF{e}_{7(7)},\MF{su}(4,4))$}{EV}{$7$}{$7$}
\LISTIV{$(\MF{e}_{7(7)},\MF{so}(6,6)+\MF{sl}(2,\BS{R}))$}{FI}{$4$}{$4$}
\LISTIV{$(\MF{e}_{7(-5)},\MF{su}(4,4))$}{EVI}{$7$}{$4$}
\LISTIV{$(\MF{e}_{7(7)},\MF{su}^{*}(8))$}{EV}{$7$}{$7$}
\LISTIV{$(\MF{e}_{7(7)},\MF{e}_{6(6)}+\BS{R})$}{CI}{$3$}{$3$}
\LISTIV{$(\MF{e}_{7(-25)},\MF{su}^{*}(8))$}{EVII}{$7$}{$3$}
\LISTIV{$(\MF{e}_{7(-5)},\MF{so}(12)+\MF{su}(2))$}{FI}{$4$}{$4$}
\LISTIV{$(\MF{e}_{7(-5)},\MF{so}^{*}(12)+\MF{sl}(2,\BS{R}))$}{FI}{$4$}{$4$}
\LISTIV{$(\MF{e}_{7(-5)},\MF{so}(8,4)+\MF{su}(2))$}{FI}{$4$}{$4$}
\LISTIV{$(\MF{e}_{7(-25)},\MF{e}_{6}+\MF{so}(2))$}{CI}{$3$}{$3$}
\LISTIV{$(\MF{e}_{7(-25)},\MF{e}_{6(-26)}+\BS{R})$}{CI}{$3$}{$3$}
\LISTIV{$(\MF{e}_{7(7)},\MF{e}_{6(2)}+\MF{so}(2))$}{CI}{$3$}{$3$}
\LISTIV{$(\MF{e}_{7(-25)},\MF{su}(6,2))$}{EVII}{$7$}{$3$}
\LISTIV{$(\MF{e}_{7(7)},\MF{so}^{*}(12)+\MF{su}(2))$}{FI}{$4$}{$4$}
\LISTIV{$(\MF{e}_{7(-5)},\MF{su}(6,2))$}{EVI}{$7$}{$4$}
\LISTIV{$(\MF{e}_{7(-5)},\MF{e}_{6(2)}+\MF{so}(2))$}{CI}{$7$}{$2$}
\LISTIV{$(\MF{e}_{7(-25)},\MF{so}^{*}(12)+\MF{su}(2))$}{FIII}{$4$}{$2$}
\end{tabular}

\begin{tabular}{|l|c|c|c|}
\hline
\multirow{2}{*}{Symmetric pair $(\MF{g},\MF{h})$} & Type of & \multirow{2}{*}{$\RANK$} & \multirow{2}{*}{$\SRANK$} \\
&$(R,\theta)$ & & \\
\hline
\hline
\LISTIV{$(\MF{e}_{7(-25)},\MF{e}_{6(-14)}+\MF{so}(2))$}{CI}{$3$}{$3$}
\LISTIV{$(\MF{e}_{7(-25)},\MF{so}(10,2)+\MF{sl}(2,\BS{R}))$}{FIII}{$4$}{$2$}
\LISTIV{$(\MF{e}_{7(-5)},\MF{e}_{6(-14)}+\MF{so}(2))$}{CI}{$3$}{$2$}
\LISTIV{$(\MF{e}_{8(8)},\MF{so}(16))$}{EVIII}{$8$}{$8$}
\LISTIV{$(\MF{e}_{8(8)},\MF{so}^{*}(16))$}{EVIII}{$8$}{$8$}
\LISTIV{$(\MF{e}_{8(8)},\MF{e}_{7(7)}+\MF{sl}(2,\BS{R}))$}{FI}{$4$}{$4$}
\LISTIV{$(\MF{e}_{8(-24)},\MF{so}^{*}(16))$}{EIX}{$8$}{$4$}
\LISTIV{$(\MF{e}_{8(8)},\MF{so}(8,8))$}{EVIII}{$8$}{$8$}
\LISTIV{$(\MF{e}_{8(-24)},\MF{e}_{7}+\MF{su}(2))$}{FI}{$4$}{$4$}
\LISTIV{$(\MF{e}_{8(-24)},\MF{e}_{7(-5)}+\MF{su}(2))$}{FI}{$4$}{$4$}
\LISTIV{$(\MF{e}_{8(-24)},\MF{so}(12,4))$}{EIX}{$8$}{$4$}
\LISTIV{$(\MF{e}_{8(8)},\MF{e}_{7(-5)}+\MF{su}(2))$}{FI}{$4$}{$4$}
\LISTIV{$(\MF{e}_{8(-24)},\MF{e}_{7(-25)}+\MF{sl}(2,\BS{R}))$}{FI}{$4$}{$4$}
\LISTIV{$(\MF{f}_{4(4)},\MF{sp}(3)+\MF{su}(2))$}{FI}{$4$}{$4$}
\LISTIV{$(\MF{f}_{4(4)},\MF{sp}(3,\BS{R})+\MF{sl}(2,\BS{R}))$}{FI}{$4$}{$4$}
\LISTIV{$(\MF{f}_{4(4)},\MF{sp}(2,1)+\MF{su}(2))$}{FI}{$4$}{$4$}
\LISTIV{$(\MF{f}_{4(4)},\MF{so}(5,4))$}{BCI}{$1$}{$1$}
\LISTIV{$(\MF{f}_{4(-20)},\MF{sp}(2,1)+\MF{su}(2))$}{FII}{$4$}{$1$}
\LISTIV{$(\MF{f}_{4(-20)},\MF{so}(9))$}{BCI}{$1$}{$1$}
\LISTIV{$(\MF{f}_{4(-20)},\MF{so}(8,1))$}{BCI}{$1$}{$1$}
\LISTIV{$(\MF{g}_{2(2)},\MF{su}(2)+\MF{su}(2))$}{G}{$2$}{$2$}
\LISTIV{$(\MF{g}_{2(2)},\MF{sl}(2,\BS{R})+\MF{sl}(2,\BS{R}))$}{G}{$2$}{$2$}
\end{tabular}
\end{multicols}
\end{table}
\end{landscape}

\begin{landscape}
\begin{table}[!!h]
\footnotesize
\renewcommand{\arraystretch}{1.1}
\contcaption{(continued)}
\centering
\begin{minipage}{\textwidth}
\small
(ii-b) $\MF{g}$ is exceptional and $\MF{g}$ is simple with a complex structure or the direct sum of two noncompact simple Lie algebras with no complex structure.
\end{minipage}
\begin{multicols}{2}
\begin{tabular}{|l|c|c|c|}
\hline
Symmetric pair $(\MF{g},\MF{h})$ & Type of $(R, \theta)$ & $\RANK$ & $\SRANK$ \\
\hline
\hline
\LISTIV{$(\MF{e}^{\BS{C}}_{6},\MF{e}_{6(-78)})$}{EI}{$6$}{$6$}
\LISTIV{$(\MF{e}^{\BS{C}}_{6},\MF{e}_{6(6)})$}{EII}{$6$}{$4$}
\LISTIV{$(\MF{e}_{6(6)}+\MF{e}_{6(6)},\MF{e}_{6(6)})$}{EI}{$6$}{$6$}
\LISTIV{$(\MF{e}^{\BS{C}}_{6},\MF{sp}(4,\BS{C}))$}{EI$+$EI}{$12$}{6}
\LISTIV{$(\MF{e}^{\BS{C}}_{6},\MF{e}_{6(2)})$}{EI}{$6$}{$6$}
\LISTIV{$(\MF{e}_{6(2)}+\MF{e}_{6(2)},\MF{e}_{6(2)})$}{EII}{$6$}{$4$}
\LISTIV{$(\MF{e}^{\BS{C}}_{6},\MF{sl}(6,\BS{C})+\MF{sl}(2,\BS{C}))$}{FI$+$FI}{$8$}{$4$}
\LISTIV{$(\MF{e}^{\BS{C}}_{6},\MF{e}_{6(-14)})$}{EI}{$6$}{$6$}
\LISTIV{$(\MF{e}_{6(-14)}+\MF{e}_{6(-14)},\MF{e}_{6(-14)})$}{EIII}{$6$}{$2$}
\LISTIV{$(\MF{e}^{\BS{C}}_{6},\MF{so}(10,\BS{C})+\BS{C})$}{BC$+$BC}{$4$}{$2$}
\LISTIV{$(\MF{e}^{\BS{C}}_{6},\MF{e}_{6(-26)})$}{EII}{$6$}{$4$}
\LISTIV{$(\MF{e}_{6(-26)}+\MF{e}_{6(-26)},\MF{e}_{6(-26)})$}{EIV}{$6$}{$2$}
\LISTIV{$(\MF{e}^{\BS{C}}_{6},\MF{f}^{\BS{C}}_{4})$}{A$+$A}{$4$}{$2$}
\LISTIV{$(\MF{e}^{\BS{C}}_{7},\MF{e}_{7(-33)})$}{EV}{$7$}{$7$}
\LISTIV{$(\MF{e}^{\BS{C}}_{7},\MF{e}_{7(7)})$}{EV}{$7$}{$7$}
\LISTIV{$(\MF{e}_{7(7)}+\MF{e}_{7(7)},\MF{e}_{7(7)})$}{EV}{$7$}{$7$}
\LISTIV{$(\MF{e}^{\BS{C}}_{7},\MF{sl}(8,\BS{C}))$}{EV$+$EV}{$14$}{$7$}
\LISTIV{$(\MF{e}^{\BS{C}}_{7},\MF{e}_{7(-5)})$}{EV}{$7$}{$7$}
\LISTIV{$(\MF{e}_{7(-5)}+\MF{e}_{7(-5)},\MF{e}_{7(-5)})$}{EVI}{$7$}{$4$}
\LISTIV{$(\MF{e}^{\BS{C}}_{7},\MF{so}(12,\BS{C})+\MF{sl}(2,\BS{C}))$}{FI$+$FI}{$8$}{$4$}
\end{tabular}

\begin{tabular}{|l|c|c|c|}
\hline
Symmetric pair $(\MF{g},\MF{h})$ & Type of $(R, \theta)$ & $\RANK$ & $\SRANK$ \\
\hline
\hline
\LISTIV{$(\MF{e}^{\BS{C}}_{7},\MF{e}_{7(-25)})$}{FI}{$7$}{$7$}
\LISTIV{$(\MF{e}_{7(-25)}+\MF{e}_{7(-25)},\MF{e}_{7(-25)})$}{EVII}{$7$}{$3$}
\LISTIV{$(\MF{e}^{\BS{C}}_{7},\MF{e}^{\BS{C}}_{6}+\BS{C})$}{C$+$C}{$6$}{$3$}
\LISTIV{$(\MF{e}^{\BS{C}}_{8},\MF{e}_{8(-248)})$}{EVIII}{$8$}{$8$}
\LISTIV{$(\MF{e}^{\BS{C}}_{8},\MF{e}_{8(8)})$}{EVIII}{$8$}{$8$}
\LISTIV{$(\MF{e}_{8(8)}+\MF{e}_{8(8)},\MF{e}_{8(8)})$}{EVIII}{$8$}{$8$}
\LISTIV{$(\MF{e}^{\BS{C}}_{8},\MF{so}(16,\BS{C}))$}{EVIII$+$EVIII}{$16$}{$8$}
\LISTIV{$(\MF{e}^{\BS{C}}_{8},\MF{e}_{8(-24)})$}{EVIII}{$8$}{$8$}
\LISTIV{$(\MF{e}_{8(-24)}+\MF{e}_{8(-24)},\MF{e}_{8(-24)})$}{EIX}{$8$}{$4$}
\LISTIV{$(\MF{e}^{\BS{C}}_{8},\MF{e}^{\BS{C}}_{7}+\MF{sl}(2,\BS{C}))$}{FI$+$FI}{$8$}{$4$}
\LISTIV{$(\MF{f}^{\BS{C}}_{4},\MF{f}_{4(-52)})$}{FI}{$4$}{$4$}
\LISTIV{$(\MF{f}^{\BS{C}}_{4},\MF{f}_{4(4)})$}{FI}{$4$}{$4$}
\LISTIV{$(\MF{f}_{4(4)}+\MF{f}_{4(4)},\MF{f}_{4(4)})$}{FI}{$4$}{$4$}
\LISTIV{$(\MF{f}^{\BS{C}}_{4},\MF{sp}(3,\BS{C})+\MF{sl}(2,\BS{C}))$}{FI$+$FI}{$8$}{$4$}
\LISTIV{$(\MF{f}^{\BS{C}}_{4},\MF{f}_{4(-20)})$}{FI}{$4$}{$4$}
\LISTIV{$(\MF{f}_{4(-20)}+\MF{f}_{4(-20)},\MF{f}_{4(-20)})$}{FII}{$4$}{$1$}
\LISTIV{$(\MF{f}^{\BS{C}}_{4},\MF{so}(9,\BS{C}))$}{BC$+$BC}{$2$}{$1$}
\LISTIV{$(\MF{g}^{\BS{C}}_{2},\MF{g}_{2(-14)})$}{G}{$2$}{$2$}
\LISTIV{$(\MF{g}^{\BS{C}}_{2},\MF{g}_{2(2)})$}{G}{$2$}{$2$}
\LISTIV{$(\MF{g}_{2(2)}+\MF{g}_{2(2)},\MF{g}_{2(2)})$}{G}{2}{2}
\LISTIV{$(\MF{g}^{\BS{C}}_{2},\MF{sl}(2,\BS{C})+\MF{sl}(2,\BS{C}))$}{G$+$G}{4}{2}
\end{tabular}
\end{multicols}
\end{table}
\end{landscape}

\begin{table}[!!h]
\renewcommand\arraystretch{1.1}
\tiny
\caption{The Satake diagram of $(R,\theta)$}\label{table.satakelistclassical}
\begin{multicols}{2}
\centering
\newcolumntype{C}{>{\centering\arraybackslash}X}
\begin{tabularx}{0.46\textwidth}{cC}
\hline
\hline
Type of & \multirow{2}{*}{Satake diagram} \\
$(R,\theta)$ & \\
\hline
\hline
\multirow{4}{*}{A$+$A} &
\multirow{4}{*}{
\begin{xy}
\ar@{-} (0,5)*+!D{\alpha_{1}}*\cir<3pt>{}="A";(5,5)*+!D{\alpha_{2}}*\cir<3pt>{}="B"
\ar@{-} "B";(8,5)
\ar@{.} (8,5);(10,5)
\ar@{-} (10,5);(13,5)*\cir<3pt>{}="C"
\ar@{-} "C";(18,5)*+!D{\alpha_{l}}*\cir<3pt>{}="D"
\ar@{-} (0,-5)*\cir<3pt>{}="AA";(5,-5)*\cir<3pt>{}="BB"
\ar@{-} "BB";(8,-5)
\ar@{.} (8,-5);(10,-5)
\ar@{-} (10,-5);(13,-5)*\cir<3pt>{}="CC"
\ar@{-} "CC";(18,-5)*\cir<3pt>{}="DD"
\ar@/_/@{<->} "A";"AA"
\ar@/_/@{<->} "B";"BB"
\ar@/_/@{<->} "C";"CC"
\ar@/_/@{<->} "D";"DD"
\end{xy}}\\\\\\\\\hline
\multirow{4}{*}{B$+$B} &
\multirow{4}{*}{
\begin{xy}
\ar@{-} (0,5)*+!D{\alpha_{1}}*\cir<3pt>{}="A";(5,5)*+!D{\alpha_{2}}*\cir<3pt>{}="B"
\ar@{-} "B";(8,5)
\ar@{.} (8,5);(10,5)
\ar@{-} (10,5);(13,5)*\cir<3pt>{}="C"
\ar@{=>} "C";(19,5)*+!D{\alpha_{l}}*\cir<3pt>{}="D"
\ar@{-} (0,-5)*\cir<3pt>{}="AA";(5,-5)*\cir<3pt>{}="BB"
\ar@{-} "BB";(8,-5)
\ar@{.} (8,-5);(10,-5)
\ar@{-} (10,-5);(13,-5)*\cir<3pt>{}="CC"
\ar@{=>} "CC";(19,-5)*\cir<3pt>{}="DD"
\ar@/_/@{<->} "A";"AA"
\ar@/_/@{<->} "B";"BB"
\ar@/_/@{<->} "C";"CC"
\ar@/_/@{<->} "D";"DD"
\end{xy}
}\\\\\\\\\hline
\multirow{5}{*}{BC$+$BC}&
\multirow{5}{*}{
\begin{xy}
\ar@{-} (0,5)*+!D{\alpha_{1}}*\cir<3pt>{}="A";(5,5)*+!D{\alpha_{2}}*\cir<3pt>{}="B"
\ar@{-} "B";(8,5)
\ar@{.} (8,5);(10,5)
\ar@{-} (10,5);(13,5)*\cir<3pt>{}="C"
\ar@{=>} "C";(19,5)*+!D{\alpha_{l}}*\cir<3pt>{}="D"
\ar@{-} (0,-5)*\cir<3pt>{}="AA";(5,-5)*\cir<3pt>{}="BB"
\ar@{-} "BB";(8,-5)
\ar@{.} (8,-5);(10,-5)
\ar@{-} (10,-5);(13,-5)*\cir<3pt>{}="CC"
\ar@{=>} "CC";(19,-5)*\cir<3pt>{}="DD"
\ar@/_/@{<->} "A";"AA"
\ar@/_/@{<->} "B";"BB"
\ar@/_/@{<->} "C";"CC"
\ar@/_/@{<->} "D";"DD"
\end{xy}
}\\\\\\\\\\\hline
\multirow{4}{*}{C$+$C}&
\multirow{4}{*}{
\begin{xy}
\ar@{-} (0,5)*+!D{\alpha_{1}}*\cir<3pt>{}="A";(5,5)*+!D{\alpha_{2}}*\cir<3pt>{}="B"
\ar@{-} "B";(8,5)
\ar@{.} (8,5);(10,5)
\ar@{-} (10,5);(13,5)*\cir<3pt>{}="C"
\ar@{<=} "C";(19,5)*+!D{\alpha_{l}}*\cir<3pt>{}="D"
\ar@{-} (0,-5)*\cir<3pt>{}="AA";(5,-5)*\cir<3pt>{}="BB"
\ar@{-} "BB";(8,-5)
\ar@{.} (8,-5);(10,-5)
\ar@{-} (10,-5);(13,-5)*\cir<3pt>{}="CC"
\ar@{<=} "CC";(19,-5)*\cir<3pt>{}="DD"
\ar@/_/@{<->} "A";"AA"
\ar@/_/@{<->} "B";"BB"
\ar@/_/@{<->} "C";"CC"
\ar@/_/@{<->} "D";"DD"
\end{xy}
}\\\\\\\\\hline
\multirow{7}{*}{D$+$D}&
\multirow{7}{*}{
\begin{xy}
\ar@{-} (0,5)*+!D{\alpha_{1}}*\cir<3pt>{}="A";(5,5)*+!D{\alpha_{2}}*\cir<3pt>{}="B"
\ar@{-} "B";(8,5)
\ar@{.} (8,5);(10,5)
\ar@{-} (10,5);(13,5)*\cir<3pt>{}="C"
\ar@{-} "C";(18,10)*\cir<3pt>{}="D"
\ar@{-} "C";(18,0)*+!L{\alpha_{l}}*\cir<3pt>{}="E"
\ar@{-} (0,-10)*\cir<3pt>{}="AA";(5,-10)*\cir<3pt>{}="BB"
\ar@{-} "BB";(8,-10)
\ar@{.} (8,-10);(10,-10)
\ar@{-} (10,-10);(13,-10)*\cir<3pt>{}="CC"
\ar@{-} "CC";(18,-5)*\cir<3pt>{}="DD"
\ar@{-} "CC";(18,-15)*\cir<3pt>{}="EE"
\ar@/_/@{<->} "A";"AA"
\ar@/_/@{<->} "B";"BB"
\ar@/_/@{<->} "C";"CC"
\ar@/^/@{<->} "D";"DD"
\ar@/^/@{<->} "E";"EE"
\end{xy}
}\\\\\\\\\\\\\\\hline
\multirow{2}{*}{AI} &
\multirow{2}{*}{
\begin{xy}
\ar@{} (0,0)*+!U{\alpha_{1}}*\cir<3pt>{}="A1"
\ar@{} (5,0)*+!U{\alpha_{2}}*\cir<3pt>{}="A2"
\ar@{} (13,0)*+!U{\alpha_{r-1}}*\cir<3pt>{}="Ar-1"
\ar@{} (18,0)*+!U{\alpha_{r}}*\cir<3pt>{}="Ar"
\ar@{-} "A1";"A2"
\ar@{-} "A2";(8,0)
\ar@{.} (8,0);(10,0)
\ar@{-} (10,0);"Ar-1"
\ar@{-} "Ar-1";"Ar"
\end{xy}
}\\\\\hline
\multirow{2}{*}{AII} &
\multirow{2}{*}{
\begin{xy}
\ar@{} (0,0)*+!U{\alpha_{1}}*\txt{\LB}="A1"
\ar@{} (5,0)*+!U{\alpha_{2}}*\cir<3pt>{}="A2"
\ar@{} (10,0)*+!U{}*\txt{\LB}="A3"
\ar@{} (18,0)*+!U{\alpha_{2l}}*\cir<3pt>{}="A2l"
\ar@{} (23,0)*+!U{\alpha_{r}}*\txt{\LB}="Ar"
\ar@{-} (0,0);"A2"
\ar@{-} "A2";(10,0)
\ar@{-} (10,0);(13,0)
\ar@{.} (13,0);(15,0)
\ar@{-} (15,0);"A2l"
\ar@{-} "A2l";(23,0)
\end{xy}
}\\\\\hline
\multirow{12}{*}{AIII} &
\multirow{12}{*}{
$\begin{cases}
\begin{xy}
\ar@{} (0,9)*+!D{\alpha_{1}}*\cir<3pt>{}="A1"
\ar@{} (5,9)*+!D{\alpha_{2}}*\cir<3pt>{}="A2"
\ar@{} (13,9)*+!D{\alpha_{l}}*\cir<3pt>{}="Al"
\ar@{} (18,9)*+!D{}*\txt{\LB}="Al+1"
\ar@{-} "A1";"A2"
\ar@{-} "A2";(8,9)
\ar@{.} (8,9);(10,9)
\ar@{-} (10,9);"Al"
\ar@{-} "Al";(18,9)
\ar@{} (18,4)*+!U{}*\txt{\LB}="Al+2"
\ar@{-} (18,9);(18,4)
\ar@{-} (18,4);(18,1)
\ar@{.} (18,1);(18,-1)
\ar@{-} (18,-1);(18,-4)
\ar@{-} (18,-4);(18,-9)
\ar@{} (18,-4)*+!U{}*\txt{\LB}="Ar-2"
\ar@{} (0,-9)*+!U{}*\cir<3pt>{}="AA1"
\ar@{} (5,-9)*+!U{}*\cir<3pt>{}="AA2"
\ar@{} (13,-9)*+!U{}*\cir<3pt>{}="AAl"
\ar@{} (18,-9)*+!U{}*\txt{\LB}="AAl+1"
\ar@{-} "AA1";"AA2"
\ar@{-} "AA2";(8,-9)
\ar@{.} (8,-9);(10,-9)
\ar@{-} (10,-9);"AAl"
\ar@{-} "AAl";(18,-9)
\ar@/_/@{<->} "A1";"AA1"
\ar@/_/@{<->} "A2";"AA2"
\ar@/_/@{<->} "Al";"AAl"
\end{xy}\\\\
\begin{xy}
\ar@{} (0,7)*+!D{\alpha_{1}}*\cir<3pt>{}="A1"
\ar@{} (5,7)*+!D{\alpha_{2}}*\cir<3pt>{}="A2"
\ar@{} (13,7)*+!D{\alpha_{l-1}}*\cir<3pt>{}="Al-1"
\ar@{-} "A1";"A2"
\ar@{-} "A2";(8,7)
\ar@{.} (8,7);(10,7)
\ar@{-} (10,7);"Al-1"
\ar@{} (0,-7)*+!U{}*\cir<3pt>{}="AA1"
\ar@{} (5,-7)*+!U{}*\cir<3pt>{}="AA2"
\ar@{} (13,-7)*+!U{}*\cir<3pt>{}="AAl-1"
\ar@{-} "AA1";"AA2"
\ar@{-} "AA2";(8,-7)
\ar@{.} (8,-7);(10,-7)
\ar@{-} (10,-7);"AAl-1"
\ar@{} (17,0)*+!DL{\alpha_{l}}*\cir<3pt>{}="Al"
\ar@{-} "Al-1";"Al"
\ar@{-} "AAl-1";"Al"
\ar@/_/@{<->} "A1";"AA1"
\ar@/_/@{<->} "A2";"AA2"
\ar@/_/@{<->} "Al-1";"AAl-1"
\end{xy}
\end{cases}$
}\\\\\\\\\\\\\\\\\\\\\\\\\hline
\multirow{2}{*}{BI} &
\multirow{2}{*}{
\begin{xy}
\ar@{} (0,0)*+!U{\alpha_{1}}*\cir<3pt>{}="A1"
\ar@{} (5,0)*+!U{\alpha_{2}}*\cir<3pt>{}="A2"
\ar@{} (13,0)*+!U{\alpha_{l}}*\cir<3pt>{}="Al"
\ar@{} (18,0)*+!U{}*\txt{\LB}="Al+1"
\ar@{} (26,0)*+!U{}*\txt{\LB}="Ar-1"
\ar@{} (32,0)*+!U{\alpha_{r}}*\txt{\LB}="Ar"
\ar@{-} "A1";"A2"
\ar@{-} "A2";(8,0)
\ar@{.} (8,0);(10,0)
\ar@{-} (10,0);"Al"
\ar@{-} "Al";(18,0)
\ar@{-} (18,0);(21,0)
\ar@{.} (21,0);(23,0)
\ar@{-} (23,0);(26,0)
\ar@{=>} (26,0);(31,0)
\end{xy}
}\\\\\hline
\multirow{2}{*}{BCI} &
\multirow{2}{*}{
\begin{xy}
\ar@{} (0,0)*+!U{\alpha_{1}}*\cir<3pt>{}="A1"
\ar@{} (5,0)*+!U{\alpha_{2}}*\cir<3pt>{}="A2"
\ar@{} (13,0)*+!U{\alpha_{l}}*\cir<3pt>{}="Al"
\ar@{} (18,0)*+!U{}*\txt{\LB}="Al+1"
\ar@{} (26,0)*+!U{}*\txt{\LB}="Ar-1"
\ar@{} (32,0)*+!U{\alpha_{r}}*\txt{\LB}="Ar"
\ar@{-} "A1";"A2"
\ar@{-} "A2";(8,0)
\ar@{.} (8,0);(10,0)
\ar@{-} (10,0);"Al"
\ar@{-} "Al";(18,0)
\ar@{-} (18,0);(21,0)
\ar@{.} (21,0);(23,0)
\ar@{-} (23,0);(26,0)
\ar@{=>} (26,0);(31,0)
\end{xy}
}\\\\\hline
\end{tabularx}

\begin{tabularx}{0.46\textwidth}{cC}
\hline
\hline
Type of & \multirow{2}{*}{Satake diagram} \\
$(R,\theta)$ & \\
\hline
\hline
\multirow{2}{*}{BCII} &
\multirow{2}{*}{
\begin{xy}
\ar@{} (0,0)*+!U{\alpha_{1}}*\cir<3pt>{}="A1"
\ar@{} (5,0)*+!U{}*\txt{\LB}="A2"
\ar@{} (13,0)*+!U{}*\txt{\LB}="Ar-1"
\ar@{} (19,0)*+!U{\alpha_{r}}*\txt{\LB}="Ar"
\ar@{-} "A1";(5,0)
\ar@{-} (5,0);(8,0)
\ar@{.} (8,0);(10,0)
\ar@{-} (10,0);(13,0)
\ar@{=>} (13,0);(18,0)
\end{xy}
}\\\\\hline
\multirow{6}{*}{BCIII} &
\multirow{6}{*}{
$\begin{cases} 
\begin{xy}
\ar@{} (0,0)*+!U{\alpha_{1}}*\txt{\LB}="A1"
\ar@{} (5,0)*+!U{\alpha_{2}}*\cir<3pt>{}="A2"
\ar@{} (10,0)*+!U{\alpha_{3}}*\txt{\LB}="A3"
\ar@{} (18,0)*+!U{\alpha_{2l}}*\cir<3pt>{}="A2l"
\ar@{} (23,0)*+!U{}*\txt{\LB}="A2l+1"
\ar@{} (28,0)*+!U{}*\txt{\LB}="A2l+2"
\ar@{} (36,0)*+!U{}*\txt{\LB}="Ar-1"
\ar@{} (42,0)*+!U{\alpha_{r}}*\txt{\LB}="Ar"
\ar@{-} (0,0);"A2"
\ar@{-} "A2";(10,0)
\ar@{-} (10,0);(13,0)
\ar@{.} (13,0);(15,0)
\ar@{-} (15,0);"A2l"
\ar@{-} "A2l";(23,0)
\ar@{-} (23,0);(28,0)
\ar@{-} (28,0);(31,0)
\ar@{.} (31,0);(33,0)
\ar@{-} (33,0);(36,0)
\ar@{=>} (36,0);(41,0)
\end{xy}\\\\
\begin{xy}
\ar@{} (0,0)*+!U{\alpha_{1}}*\txt{\LB}="A1"
\ar@{} (5,0)*+!U{\alpha_{2}}*\cir<3pt>{}="A2"
\ar@{} (10,0)*+!U{}*\txt{\LB}="A3"
\ar@{} (18,0)*+!U{}*\cir<3pt>{}="A2l-2"
\ar@{} (23,0)*+!U{}*\txt{\LB}="A2l-1"
\ar@{} (29,0)*+!U{\alpha_{2l}}*\cir<3pt>{}="A2l"
\ar@{-} (0,0);"A2"
\ar@{-} "A2";(10,0)
\ar@{-} (10,0);(13,0)
\ar@{.} (13,0);(15,0)
\ar@{-} (15,0);"A2l-2"
\ar@{-} "A2l-2";(23,0)
\ar@{=>} (23,0);"A2l"
\end{xy}
\end{cases}$
}\\\\\\\\\\\\\hline
\multirow{2}{*}{CI} &
\multirow{2}{*}{
\begin{xy}
\ar@{} (0,0)*+!U{\alpha_{1}}*\cir<3pt>{}="A1"
\ar@{} (5,0)*+!U{\alpha_{2}}*\cir<3pt>{}="A2"
\ar@{} (13,0)*+!U{\alpha_{l}}*\cir<3pt>{}="Al"
\ar@{} (18,0)*+!U{}*\txt{\LB}="Al+1"
\ar@{} (26,0)*+!U{}*\txt{\LB}="Ar-1"
\ar@{} (32,0)*+!U{\alpha_{r}}*\txt{\LB}="Ar"
\ar@{-} "A1";"A2"
\ar@{-} "A2";(8,0)
\ar@{.} (8,0);(10,0)
\ar@{-} (10,0);"Al"
\ar@{-} "Al";(18,0)
\ar@{-} (18,0);(21,0)
\ar@{.} (21,0);(23,0)
\ar@{-} (23,0);(26,0)
\ar@{<=} (27,0);(32,0)
\end{xy}
}\\\\\hline
\multirow{2}{*}{CII} &
\multirow{2}{*}{
\begin{xy}
\ar@{} (0,0)*+!U{\alpha_{1}}*\cir<3pt>{}="A1"
\ar@{} (5,0)*+!U{}*\txt{\LB}="A2"
\ar@{} (13,0)*+!U{}*\txt{\LB}="Ar-1"
\ar@{} (19,0)*+!U{\alpha_{r}}*\txt{\LB}="Ar"
\ar@{-} "A1";(5,0)
\ar@{-} (5,0);(8,0)
\ar@{.} (8,0);(10,0)
\ar@{-} (10,0);(13,0)
\ar@{<=} (14,0);(19,0)
\end{xy}
}\\\\\hline
\multirow{6}{*}{CIII} &
\multirow{6}{*}{
$\begin{cases} 
\begin{xy}
\ar@{} (0,0)*+!U{\alpha_{1}}*\txt{\LB}="A1"
\ar@{} (5,0)*+!U{\alpha_{2}}*\cir<3pt>{}="A2"
\ar@{} (10,0)*+!U{\alpha_{3}}*\txt{\LB}="A3"
\ar@{} (18,0)*+!U{\alpha_{2l}}*\cir<3pt>{}="A2l"
\ar@{} (23,0)*+!U{}*\txt{\LB}="A2l+1"
\ar@{} (28,0)*+!U{}*\txt{\LB}="A2l+2"
\ar@{} (36,0)*+!U{}*\txt{\LB}="Ar-1"
\ar@{} (42,0)*+!U{\alpha_{r}}*\txt{\LB}="Ar"
\ar@{-} (0,0);"A2"
\ar@{-} "A2";(10,0)
\ar@{-} (10,0);(13,0)
\ar@{.} (13,0);(15,0)
\ar@{-} (15,0);"A2l"
\ar@{-} "A2l";(23,0)
\ar@{-} (23,0);(28,0)
\ar@{-} (28,0);(31,0)
\ar@{.} (31,0);(33,0)
\ar@{-} (33,0);(36,0)
\ar@{<=} (37,0);(42,0)
\end{xy}\\\\
\begin{xy}
\ar@{} (0,0)*+!U{\alpha_{1}}*\txt{\LB}="A1"
\ar@{} (5,0)*+!U{\alpha_{2}}*\cir<3pt>{}="A2"
\ar@{} (10,0)*+!U{}*\txt{\LB}="A3"
\ar@{} (18,0)*+!U{}*\cir<3pt>{}="A2l-2"
\ar@{} (23,0)*+!U{}*\txt{\LB}="A2l-1"
\ar@{} (29,0)*+!U{\alpha_{2l}}*\cir<3pt>{}="A2l"
\ar@{-} (0,0);"A2"
\ar@{-} "A2";(10,0)
\ar@{-} (10,0);(13,0)
\ar@{.} (13,0);(15,0)
\ar@{-} (15,0);"A2l-2"
\ar@{-} "A2l-2";(23,0)
\ar@{<=} (24,0);"A2l"
\end{xy}
\end{cases}$
}\\\\\\\\\\\\\hline
\multirow{12}{*}{DI} &
\multirow{12}{*}{
$\begin{cases}
\begin{xy}
\ar@{} (0,0)*+!U{\alpha_{1}}*\cir<3pt>{}="A1"
\ar@{} (8,0)*+!U{\alpha_{l-2}}*\cir<3pt>{}="Al-2"
\ar@{} (12,6)*+!L{\alpha_{l-1}}*\cir<3pt>{}="Al-1"
\ar@{} (12,-6)*+!L{\alpha_{l}}*\cir<3pt>{}="Al"
\ar@{-} "A1";(3,0)
\ar@{.} (3,0);(5,0)
\ar@{-} (5,0);"Al-2"
\ar@{-} "Al-2";"Al-1"
\ar@{-} "Al-2";"Al"
\end{xy}
\\\\
\begin{xy}
\ar@{} (0,0)*+!U{\alpha_{1}}*\cir<3pt>{}="A1"
\ar@{} (8,0)*+!U{\alpha_{l-1}}*\cir<3pt>{}="Al-1"
\ar@{} (12,6)*+!L{\alpha_{l}}*\cir<3pt>{}="Al"
\ar@{} (12,-6)*+!L{\alpha_{l+1}}*\cir<3pt>{}="Al+1"
\ar@{-} "A1";(3,0)
\ar@{.} (3,0);(5,0)
\ar@{-} (5,0);"Al-1"
\ar@{-} "Al-1";"Al"
\ar@{-} "Al-1";"Al+1"
\ar@/^/@{<->} "Al";"Al+1"
\end{xy}\\\\
\begin{xy}
\ar@{} (0,0)*+!U{\alpha_{1}}*\cir<3pt>{}="A1"
\ar@{} (5,0)*+!U{}*\cir<3pt>{}="A2"
\ar@{} (13,0)*+!U{\alpha_{l}}*\cir<3pt>{}="Al"
\ar@{} (18,0)*+!U{}*\txt{\LB}="Al+1"
\ar@{} (26,0)*+!U{}*\txt{\LB}="Ar-2"
\ar@{} (30,6)*+!U{}*\txt{\LB}="Ar-1"
\ar@{} (30,-6)*+!U{}*\txt{\LB}="Ar"
\ar@{-} "A1";"A2"
\ar@{-} "A2";(8,0)
\ar@{.} (8,0);(10,0)
\ar@{-} (10,0);"Al"
\ar@{-} "Al";(18,0)
\ar@{-} (18,0);(21,0)
\ar@{.} (21,0);(23,0)
\ar@{-} (23,0);(26,0)
\ar@{-} (26,0);(30,6)
\ar@{-} (26,0);(30,-6)
\end{xy}
\end{cases}$}\\\\\\\\\\\\\\\\\\\\\\\\\\\hline
\multirow{4}{*}{DII} &
\multirow{4}{*}{
\begin{xy}
\ar@{} (0,0)*+!U{\alpha_{1}}*\cir<3pt>{}="A1"
\ar@{} (5,0)*+!U{}*\txt{\LB}="A2"
\ar@{} (13,0)*+!U{}*\txt{\LB}="Ar-2"
\ar@{} (17,6)*+!U{}*\txt{\LB}="Ar-1"
\ar@{} (17,-6)*+!U{}*\txt{\LB}="Ar"
\ar@{-} "A1";(5,0)
\ar@{-} (5,0);(8,0)
\ar@{.} (8,0);(10,0)
\ar@{-} (10,0);(13,0)
\ar@{-} (13,0);(17,6)
\ar@{-} (13,0);(17,-6)
\end{xy}
}\\\\\\\\\hline
\multirow{9}{*}{DIII} &
\multirow{9}{*}{
$\begin{cases}
\begin{xy}
\ar@{}  (0,0)*+!U{\alpha_{1}}*\txt{\LB}="A1"
\ar@{} (5,0)*+!U{\alpha_{2}}*\cir<3pt>{}="A2"
\ar@{}  (10,0)*+!U{}*\txt{\LB}="A3"
\ar@{} (18,0)*+!U{\alpha_{r-2}}*\cir<3pt>{}="Ar-2"
\ar@{}  (22,6)*+!L{\alpha_{r-1}}*\txt{\LB}="Ar-1"
\ar@{} (22,-6)*+!L{\alpha_{r}}*\cir<3pt>{}="Ar"
\ar@{-} (0,0);"A2"
\ar@{-} "A2";(10,0)
\ar@{-} (10,0);(13,0)
\ar@{.} (13,0);(15,0)
\ar@{-} (15,0);"Ar-2"
\ar@{-} "Ar-2";(22,6)
\ar@{-} "Ar-2";"Ar"
\end{xy}\\\\
\begin{xy}
\ar@{}  (0,0)*+!U{\alpha_{1}}*\txt{\LB}="A1"
\ar@{} (5,0)*+!U{\alpha_{2}}*\cir<3pt>{}="A2"
\ar@{}  (10,0)*+!U{}*\txt{\LB}="A3"
\ar@{} (18,0)*+!U{\alpha_{r-2}}*\cir<3pt>{}="Ar-2"
\ar@{}  (22,6)*+!L{\alpha_{r-1}}*\cir<3pt>{}="Ar-1"
\ar@{} (22,-6)*+!L{\alpha_{r}}*\cir<3pt>{}="Ar"
\ar@{-} (0,0);"A2"
\ar@{-} "A2";(10,0)
\ar@{-} (10,0);(13,0)
\ar@{.} (13,0);(15,0)
\ar@{-} (15,0);"Ar-2"
\ar@{-} "Ar-2";"Ar-1"
\ar@{-} "Ar-2";"Ar"
\ar@/_/@{<->} "Ar";"Ar-1"
\end{xy}
\end{cases}$}\\\\\\\\\\\\\\\\\\\hline
\end{tabularx}
\end{multicols}
\end{table}

\begin{table}[!!h]
\tiny
\contcaption{(continued)}
\begin{multicols}{2}
\newcolumntype{C}{>{\centering\arraybackslash}X}
\begin{tabularx}{0.46\textwidth}{cC}
\hline
\hline
Type of & \multirow{2}{*}{Satake diagram} \\
$(R,\theta)$ & \\
\hline
\hline
\multirow{7}{*}{EI$+$EI}&
\multirow{7}{*}{
\begin{xy}
\ar@{-} (0,5)*\cir<3pt>{}="X";(5,5)*\cir<3pt>{}="A"
\ar@{-} "A";(10,5)*\cir<3pt>{}="B"
\ar@{-} "B";(15,5)*\cir<3pt>{}="C"
\ar@{-} "C";(20,5)*\cir<3pt>{}="D"
\ar@{-} "B";(10,10)*\cir<3pt>{}="Y" 
\ar@{-} "A";(10,5)*\cir<3pt>{}="B"
\ar@{-} "B";(15,5)*\cir<3pt>{}="C"
\ar@{-} "C";(20,5)*\cir<3pt>{}="D"
\ar@{-} "B";(10,10)*\cir<3pt>{}="Y" 
\ar@{-} (0,-7)*\cir<3pt>{}="XX";(5,-7)*\cir<3pt>{}="AA"
\ar@{-} "AA";(10,-7)*\cir<3pt>{}="BB"
\ar@{-} "BB";(15,-7)*\cir<3pt>{}="CC"
\ar@{-} "CC";(20,-7)*\cir<3pt>{}="DD"
\ar@{-} "BB";(10,-2)*\cir<3pt>{}="YY" 
\ar@{-} "AA";(10,-7)*\cir<3pt>{}="BB"
\ar@{-} "BB";(15,-7)*\cir<3pt>{}="CC"
\ar@{-} "CC";(20,-7)*\cir<3pt>{}="DD"
\ar@{-} "BB";(10,-2
)*\cir<3pt>{}="YY" 
\ar@/_/@{<->} "A";"AA"
\ar@/_/@{<->} "B";"BB"
\ar@/_/@{<->} "C";"CC"
\ar@/_/@{<->} "D";"DD"
\ar@/_/@{<->} "X";"XX"
\ar@/^/@{<->} "Y";"YY"
\end{xy}
}\\\\\\\\\\\\\\\hline
\multirow{7}{*}{EV$+$EV}&
\multirow{7}{*}{
\begin{xy}
\ar@{-} (0,5)*\cir<3pt>{}="X";(5,5)*\cir<3pt>{}="A"
\ar@{-} "A";(10,5)*\cir<3pt>{}="B"
\ar@{-} "B";(15,5)*\cir<3pt>{}="C"
\ar@{-} "C";(20,5)*\cir<3pt>{}="D"
\ar@{-} "D";(25,5)*\cir<3pt>{}="E"
\ar@{-} "C";(15,10)*\cir<3pt>{}="Y" 
\ar@{-} (0,-7)*\cir<3pt>{}="XX";(5,-7)*\cir<3pt>{}="AA"
\ar@{-} "AA";(10,-7)*\cir<3pt>{}="BB"
\ar@{-} "BB";(15,-7)*\cir<3pt>{}="CC"
\ar@{-} "CC";(20,-7)*\cir<3pt>{}="DD"
\ar@{-} "DD";(25,-7)*\cir<3pt>{}="EE"
\ar@{-} "CC";(15,-2)*\cir<3pt>{}="YY"
\ar@/_/@{<->} "A";"AA"
\ar@/_/@{<->} "B";"BB"
\ar@/_/@{<->} "C";"CC"
\ar@/_/@{<->} "D";"DD"
\ar@/_/@{<->} "E";"EE"
\ar@/_/@{<->} "X";"XX"
\ar@/^/@{<->} "Y";"YY"
\end{xy}
}\\\\\\\\\\\\\\\hline
\multirow{7}{*}{EVIII$+$EVIII}&
\multirow{7}{*}{
\begin{xy}
\ar@{-} (0,5)*\cir<3pt>{}="X";(5,5)*\cir<3pt>{}="A"
\ar@{-} "A";(10,5)*\cir<3pt>{}="B"
\ar@{-} "B";(15,5)*\cir<3pt>{}="C"
\ar@{-} "C";(20,5)*\cir<3pt>{}="D"
\ar@{-} "D";(25,5)*\cir<3pt>{}="E"
\ar@{-} "E";(30,5)*\cir<3pt>{}="F"
\ar@{-} "D";(20,10)*\cir<3pt>{}="Y" 
\ar@{-} (0,-7)*\cir<3pt>{}="XX";(5,-7)*\cir<3pt>{}="AA"
\ar@{-} "AA";(10,-7)*\cir<3pt>{}="BB"
\ar@{-} "BB";(15,-7)*\cir<3pt>{}="CC"
\ar@{-} "CC";(20,-7)*\cir<3pt>{}="DD"
\ar@{-} "DD";(25,-7)*\cir<3pt>{}="EE"
\ar@{-} "EE";(30,-7)*\cir<3pt>{}="FF"
\ar@{-} "DD";(20,-2)*\cir<3pt>{}="YY" 
\ar@/_/@{<->} "A";"AA"
\ar@/_/@{<->} "B";"BB"
\ar@/_/@{<->} "C";"CC"
\ar@/_/@{<->} "D";"DD"
\ar@/_/@{<->} "E";"EE"
\ar@/_/@{<->} "F";"FF"
\ar@/_/@{<->} "X";"XX"
\ar@/^/@{<->} "Y";"YY"
\end{xy}
}\\\\\\\\\\\\\\\hline
\multirow{6}{*}{FI$+$FI}&
\multirow{6}{*}{
\begin{xy}
\ar@{-} (0,5)*\cir<3pt>{}="A";(5,5)*\cir<3pt>{}="B"
\ar@{=>} "B";(12,5)*\cir<3pt>{}="C"
\ar@{-}  "C";(17,5)*\cir<3pt>{}="D"
\ar@{-} (0,-5)*\cir<3pt>{}="AA";(5,-5)*\cir<3pt>{}="BB"
\ar@{=>} "BB";(12,-5)*\cir<3pt>{}="CC"
\ar@{-}  "CC";(17,-5)*\cir<3pt>{}="DD"
\ar@/_/@{<->} "A";"AA"
\ar@/_/@{<->} "B";"BB"
\ar@/_/@{<->} "C";"CC"
\ar@/_/@{<->} "D";"DD"
\end{xy}
}\\\\\\\\\\\\\hline
\multirow{6}{*}{G$+$G}&
\multirow{6}{*}{
\begin{xy}
\ar@3{->} (0,5)*\cir<3pt>{}="A";(7,5)*\cir<3pt>{}="B"
\ar@3{->} (0,-5)*\cir<3pt>{}="AA";(7,-5)*\cir<3pt>{}="BB"
\ar@/_/@{<->} "A";"AA"
\ar@/_/@{<->} "B";"BB"
\end{xy}
}\\\\\\\\\\\\\hline
\multirow{5}{*}{EI} &
\multirow{5}{*}{
\begin{xy}
\ar@{-} (0,1)*\cir<3pt>{};(5,1)*\cir<3pt>{}="A"
\ar@{-} "A";(10,1)*\cir<3pt>{}="B"
\ar@{-} "B";(15,1)*\cir<3pt>{}="C"
\ar@{-} "C";(20,1)*\cir<3pt>{}="D"
\ar@{-} "B";(10,6)*\cir<3pt>{} 
\end{xy}
}\\\\\\\\\\\hline
\multirow{5}{*}{EII} &
\multirow{5}{*}{
\begin{xy}
\ar@{-} (0,1)*+!D{\alpha_{6}}*\cir<3pt>{}="A";(5,1)*+!D{\alpha_{5}}*\cir<3pt>{}="B"
\ar@{-} "B";(10,1)*+!DL{\!\!\!\alpha_{4}}*\cir<3pt>{}="C"
\ar@{-} "C";(15,1)*+!D{\alpha_{3}}*\cir<3pt>{}="D"
\ar@{-} "D";(20,1)*+!D{\alpha_{1}}*\cir<3pt>{}="E"
\ar@{-} "C";(10,6)*+!D{\alpha_{2}}*\cir<3pt>{}
\ar@/_4mm/@{<->} "A";"E"
\ar@/_/@{<->} "B";"D"
\end{xy}
}\\\\\\\\\\\hline
\multirow{5}{*}{EIII} &
\multirow{5}{*}{
\begin{xy}
\ar@{-} (0,1)*+!D{\alpha_{6}}*\cir<3pt>{}="A";(5,1)
\ar@{}  (5,0.8)*+!D{\alpha_{5}}*\txt{\LB}
\ar@{}  (10,0.8)*+!DL{\!\!\!\alpha_{4}}*\txt{\LB}
\ar@{}  (15,0.8)*+!D{\alpha_{3}}*\txt{\LB}
\ar@{-} (5,1);(10,1)
\ar@{-} (10,1);(15,1)
\ar@{-} (15,1);(20,1)*+!D{\alpha_{1}}*\cir<3pt>{}="E"
\ar@{-} (10,1);(10,6)*+!D{\alpha_{2}}*\cir<3pt>{}
\ar@/_4mm/@{<->} "A";"E"
\end{xy}
}\\\\\\\\\\\hline
\end{tabularx}

\begin{tabularx}{0.46\textwidth}{cC}
\hline
\hline
Type of & \multirow{2}{*}{Satake diagram} \\
$(R,\theta)$ & \\
\hline
\hline
\multirow{5}{*}{EIV} &
\multirow{5}{*}{
\begin{xy}
\ar@{-} (0,1)*+!U{\alpha_{6}}*\cir<3pt>{}="A";(5,1)
\ar@{}  (5,0.8)*\txt{\LB}
\ar@{}  (10,0.8)*\txt{\LB}
\ar@{}  (15,0.8)*\txt{\LB}
\ar@{-} (5,1);(10,1)
\ar@{-} (10,1);(15,1)
\ar@{-} (15,1);(20,1)*+!U{\alpha_{1}}*\cir<3pt>{}="E"
\ar@{-} (10,1);(10,6)
\ar@{}  (10,5.8)*\txt{\LB}
\end{xy}
}\\\\\\\\\\\hline
\multirow{5}{*}{EV} &
\multirow{5}{*}{
\begin{xy}
\ar@{-} (0,1)*\cir<3pt>{};(5,1)*\cir<3pt>{}="A"
\ar@{-} "A";(10,1)*\cir<3pt>{}="B"
\ar@{-} "B";(15,1)*\cir<3pt>{}="C"
\ar@{-} "C";(20,1)*\cir<3pt>{}="D"
\ar@{-} "D";(25,1)*\cir<3pt>{}="E"
\ar@{-} "C";(15,6)*\cir<3pt>{} 
\end{xy}
}\\\\\\\\\\\hline
\multirow{6}{*}{EVI} &
\multirow{6}{*}{
\begin{xy}
\ar@{} (0,0.8)*+!U{\alpha_{7}}*\txt{\LB}
\ar@{-} (0,1);(5,1)*+!U{\alpha_{6}}*\cir<3pt>{}="A"
\ar@{-} "A";(10,1)
\ar@{}  (10,1)*+!U{\alpha_{5}}*\txt{\LB}
\ar@{-} (10,1);(15,1)*+!U{\alpha_{4}}*\cir<3pt>{}="C"
\ar@{-} "C";(20,1)*+!U{\alpha_{3}}*\cir<3pt>{}="D"
\ar@{-} "D";(25,1)*+!U{\alpha_{1}}*\cir<3pt>{}="E"
\ar@{-} "C";(15,6)
\ar@{} (15,5.8)*+!D{\alpha_{2}}*\txt{\LB}
\end{xy}
}\\\\\\\\\\\\\hline
\multirow{6}{*}{EVII} &
\multirow{6}{*}{
\begin{xy}
\ar@{-} (0,1)*+!U{\alpha_{7}}*\cir<3pt>{};(5,1)*+!U{\alpha_{6}}*\cir<3pt>{}="A"
\ar@{-} "A";(10,1)
\ar@{-} (10,1);(15,1)
\ar@{-} (15,1);(20,1)
\ar@{-} (25,1);(25,1)
\ar@{-} (15,1);(15,6)
\ar@{-} (20,1);(25,1)*+!U{\alpha_{1}}*\cir<3pt>{}
\ar@{}  (10,0.8)*+!U{\alpha_{5}}*\txt{\LB}
\ar@{}  (15,0.8)*+!U{\alpha_{4}}*\txt{\LB}
\ar@{}  (20,0.8)*+!U{\alpha_{3}}*\txt{\LB}
\ar@{}  (15,5.8)*+!D{\alpha_{2}}*\txt{\LB}
\end{xy}
}\\\\\\\\\\\\\hline
\multirow{5}{*}{EVIII} &
\multirow{5}{*}{
\begin{xy}
\ar@{-} (0,1)*\cir<3pt>{};(5,1)*\cir<3pt>{}="A"
\ar@{-} "A";(10,1)*\cir<3pt>{}="B"
\ar@{-} "B";(15,1)*\cir<3pt>{}="C"
\ar@{-} "C";(20,1)*\cir<3pt>{}="D"
\ar@{-} "D";(25,1)*\cir<3pt>{}="E"
\ar@{-} "E";(30,1)*\cir<3pt>{}
\ar@{-} "D";(20,6)*\cir<3pt>{} 
\end{xy}
}\\\\\\\\\\\hline
\multirow{6}{*}{EIX} &
\multirow{6}{*}{
\begin{xy}
\ar@{-} (0,1)*+!U{\alpha_{8}}*\cir<3pt>{};(5,1)*+!U{\alpha_{7}}*\cir<3pt>{}="A"
\ar@{-} "A";(10,1)*+!U{\alpha_{6}}*\cir<3pt>{}="B"
\ar@{-} "B";(15,1)
\ar@{-} (15,1);(20,1)
\ar@{-} (20,1);(25,1)
\ar@{-} (25,1);(30,1)*+!U{\alpha_{1}}*\cir<3pt>{}
\ar@{-} (20,1);(20,6)
\ar@{}  (15,0.8)*+!U{\alpha_{5}}*\txt{\LB}
\ar@{}  (20,0.8)*+!U{\alpha_{4}}*\txt{\LB}
\ar@{}  (25,0.8)*+!U{\alpha_{3}}*\txt{\LB}
\ar@{}  (20,5.8)*+!D{\alpha_{2}}*\txt{\LB}
\end{xy}
}\\\\\\\\\\\\\hline
\multirow{4}{*}{FI} &
\multirow{4}{*}{
\begin{xy}
\ar@{-} (0,1)*\cir<3pt>{};(5,1)*\cir<3pt>{}="A"
\ar@{=>} "A";(12,1)*\cir<3pt>{}="B"
\ar@{-}  "B";(17,1)*\cir<3pt>{}
\end{xy}
}\\\\\\\\\hline
\multirow{4}{*}{FII} &
\multirow{4}{*}{
\begin{xy}
\ar@{-} (0,1);(5,1)
\ar@{=>} (5,1);(10.8,1)
\ar@{-}  (12,1);(17,1)*+!U{\alpha_{4}}*\cir<3pt>{}
\ar@{} (0,0.8)*+!U{\alpha_{1}}*\txt{\LB}
\ar@{} (5,0.8)*+!U{\alpha_{2}}*\txt{\LB}
\ar@{} (12,0.8)*+!U{\alpha_{3}}*\txt{\LB}
\end{xy}
}\\\\\\\\\hline
\multirow{4}{*}{FIII} &
\multirow{4}{*}{
\begin{xy}
\ar@{-} (0,1)*+!U{\alpha_{1}}*\cir<3pt>{};(5,1)
\ar@{=>} (5,1);(10.8,1)
\ar@{-}  (12,1);(17,1)*+!U{\alpha_{4}}*\cir<3pt>{}
\ar@{} (5,0.8)*+!U{\alpha_{2}}*\txt{\LB}
\ar@{} (12,0.8)*+!U{\alpha_{3}}*\txt{\LB}
\end{xy}
}\\\\\\\\\hline
\multirow{3}{*}{G} &
\multirow{3}{*}{
\begin{xy}
\ar@3{->} (0,1)*\cir<3pt>{};(7,1)*\cir<3pt>{}
\end{xy}
}\\\\\\\hline
\end{tabularx}
\end{multicols}
\end{table}


\clearpage



\begin{thebibliography}{99}

\def\cprime{$'$}

\bibitem{MR2722116}
\textsc{H.\ Anciaux},
\textrm{Minimal submanifolds in pseudo-{R}iemannian geometry},
World Scientific Publishing Co.\ Pte.\ Ltd., Hackensack, NJ,
2011.

\bibitem{B3}
\textsc{K.\ Baba},
\textrm{Local orbit types of the isotropy representations for semisimple pseudo-Riemannian symmetric spaces},
Differential Geom.\ Appl.\ \textbf{38} (2015), 124--150.

\bibitem{MR0104763}
\textsc{M.\ Berger},
\textrm{Les espaces sym\'etriques noncompacts},
Ann.\ Sci.\ \'Ecole Norm.\allowbreak\ Sup.\ \textbf{74} (1957), 85--177.

\bibitem{MR0930601}
\textsc{J.\ Hahn},
\textrm{Isotropy representations of semisimple symmetric spaces and homogeneous hypersurfaces},
J.\ Math.\ Soc.\ Japan,
\textbf{40} (1988),
271--288.

\bibitem{MR0666108}
\textsc{R.\ Harvey and H.\ B.\ Lawson},
\textrm{Calibrated geometries},
Acta Math., \textbf{148}, (1982), 47--157.

\bibitem{MR1834454}
\textsc{S.\ Helgason},
\textrm{Differential geometry, {L}ie groups, and symmetric spaces},
Graduate Studies in Mathematics, American Mathematical Society, Providence, Rhode Island, 2001.

\bibitem{MR2752433}
\textsc{O.\ Ikawa},
\textrm{The geometry of symmetric triad and orbit spaces of Hermann actions},
J.\ Math.\ Soc.\ Japan \textbf{63} (2011), 79--136.

\bibitem{MR2532897}
\textsc{O.\ Ikawa, T.\ Sakai and H.\ Tasaki},
\textrm{Weakly reflective submanifolds and austere submanifolds},
J.\ Math.\ Soc.\ Japan \textbf{61} (2009), 437--481.

\bibitem{MR3178477}
\textsc{N.\ Koike},
\textrm{Examples of certain kind of minimal orbits of Hermann actions},
Hokkaido Math.\ J.\ \textbf{43} (2014), 21--42.

\bibitem{MR0239005}
\textsc{O.\ Loos},
\textrm{Symmetric spaces.\ {I}: {G}eneral theory, {II}: {C}ompact spaces and classification},
W.\ A.\ Benjamin, Inc., New York-Amsterdam, 1969.

\bibitem{MR567427}
\textsc{T.\ Oshima and T.\ Matsuki},
\textrm{Orbits on affine symmetric spaces under the action of the isotropy subgroups},
J.\ Math.\ Soc.\ Japan, \textbf{32} (1980), 399--414.

\bibitem{MR810638}
\textsc{T.\ Oshima and J.\ Sekiguchi},
\textrm{The restricted root system of a semisimple symmetric pair},
Adv.\ Stud.\ Math.\ \textbf{4} (1984), 433--497.

\bibitem{MR518716}
\textsc{W.\ Rossmann},
\textrm{The structure of semisimple symmetric spaces},
Canad.\ J.\ Math.\ \textbf{31} (1979), 157--180.

\bibitem{MR1280269}
\textsc{M.\ Takeuchi},
\textrm{Modern spherical functions},
Translations of Mathematical Monographs, 135, American Mathematical
Society, Providence, RI, 1994.

\bibitem{MR0498999}
\textsc{G.\ Warner},
\textrm{Harmonic analysis on semi-simple {L}ie groups. {I}},
Springer-Verlag, New York, 1972.

\end{thebibliography}
\end{document}